\documentclass[11pt]{amsart}
\usepackage[toc,page]{appendix}
\usepackage{amsmath,amssymb,latexsym}
\usepackage[small]{caption}
\usepackage{graphicx,color,mathrsfs,tikz}
\usepackage{subfigure,color}
\usepackage{cite}
\usepackage[colorlinks=true,urlcolor=blue,
citecolor=red,linkcolor=blue,linktocpage,pdfpagelabels,
bookmarksnumbered,bookmarksopen]{hyperref}
\usepackage[italian,english]{babel}
\usepackage{enumitem}
\usepackage[left=2.4cm,right=2.4cm,top=2.9cm,bottom=2.9cm]{geometry}
\usepackage[hyperpageref]{backref}
\usepackage[colorinlistoftodos,prependcaption]{todonotes}

\numberwithin{equation}{section}
\newtheorem{theorem}{Theorem}[section]
\newtheorem{proposition}[theorem]{Proposition}
\newtheorem{lemma}[theorem]{Lemma}

\newtheorem{definition}[theorem]{Definition}
\theoremstyle{definition}

\newtheorem{remark}[theorem]{Remark}

\renewcommand{\epsilon}{\eps}

\newcommand{\D}{{\mathcal D}}

\renewcommand{\L}{{\mathcal L}}
\newcommand{\C}{{\mathcal C}}

\newcommand{\R}{{\mathbb R}}

\newcommand{\eps}{\varepsilon}

\newcommand{\M}{\mathscr{M}}

\newcommand{\pnorm}[2][]{\if #1'' \left|#2\right|_p \else \left|#2\right|_{#1} \fi}

\newcommand{\Ds}[1]{(-\Delta)^{#1}}
\newcommand{\frin}{(-\Delta)^{s}_\infty}
\newcommand{\frp}{(-\Delta)^{s}_p}

\renewcommand{\theta}{\vartheta}

\newcommand{\Sf}{{\mathbb S^{n-1}}}
\newcommand{\Rn}{{\mathbb R^{n}}}
\newcommand{\eqlab}[1]{\begin{equation}  \begin{aligned}#1 \end{aligned}\end{equation}} 
\newcommand{\bgs}[1]{\begin{equation*} \begin{aligned}#1\end{aligned}\end{equation*}} 
 \newcommand{\syslab}[2] []  {\begin{equation}#1  \left\{\begin{aligned}#2\end{aligned}\right.\end{equation}} 
  \newcommand{\sys}[2][]{\begin{equation*}#1  \left\{\begin{aligned}#2\end{aligned}\right.\end{equation*}}

\makeatletter

\def\Xint#1{\mathchoice
{\XXint\displaystyle\textstyle{#1}}%
{\XXint\textstyle\scriptstyle{#1}}%
{\XXint\scriptstyle\scriptscriptstyle{#1}}%
{\XXint\scriptscriptstyle\scriptscriptstyle{#1}}%
\!\int}
\def\XXint#1#2#3{{\setbox0=\hbox{$#1{#2#3}{\int}$ }
\vcenter{\hbox{$#2#3$ }}\kern-.6\wd0}}

\def\dashint{\Xint-}

\title[Asymptotic expansion for nonlinear nonlocal operators]{An asymptotic expansion for the fractional $p$-Laplacian \\ and for gradient dependent nonlocal operators}

\author[C. Bucur]{Claudia Bucur}
\author[M.\ Squassina]{Marco Squassina}

\address[C. Bucur]{Dipartimento di Scienza ed Alta Tecnologia \newline\indent Universit\`a degli Studi dell'Insubria 
\&
\newline\indent
Riemann International School of Mathematics
\newline\indent
 Villa Toeplitz - Via G.B. Vico 46 - Varese, Italy. 
}
\email{claudia.bucur@aol.com, claudiadalia.bucur@uninsubria.it}

\address[M.\ Squassina]{Dipartimento di Matematica e Fisica \newline\indent
	Universit\`a Cattolica del Sacro Cuore \newline\indent
	Via dei Musei 41, I-25121 Brescia, Italy}
\email{marco.squassina@unicatt.it}
\thanks{The authors are members
	of {\em Gruppo Nazionale per l'Analisi Ma\-te\-ma\-ti\-ca, la Probabilit\`a e le loro Applicazioni} (GNAMPA) of the {\em Istituto Nazionale di Alta Matematica} (INdAM). 
	The first author is supported by the INdAM Starting Grant
``PDEs, free boundaries, nonlocal equations and applications''.
\newline\indent
The authors thank Hoai-Minh Nguyen for very fruitful discussions on the subject, carried out while visiting \'Ecole Polytechnique F\'ed\'erale de Lausanne. The hosting institution is gratefully acknowledged. \textcolor{black}{The authors also thank F\'elix Del Teso, who kindly pointed out some imprecisions in a preliminary version of this paper.}
}



\begin{document}

\begin{abstract}
	Mean value formulas are of great importance in the theory of partial differential equations: many very useful results are drawn, for instance, from the well known equivalence between harmonic functions and mean value properties. In the nonlocal setting of  fractional harmonic functions, such an equivalence still holds, and many applications are now-days available. The nonlinear case, corresponding to the  $p$-Laplace operator, has also been recently investigated, whereas the validity of a nonlocal, nonlinear, counterpart remains an open problem. In this paper, we propose a formula for the \emph{nonlocal, nonlinear mean value kernel}, by means of which we obtain an asymptotic representation formula for harmonic functions in the viscosity sense,
	with respect to the fractional (variational) $p$-Laplacian (for $p\geq 2$) and to other gradient dependent nonlocal operators.  \end{abstract}
\maketitle
%

\medskip

\section{Introduction}
One of the most famous basic facts of partial differential equations is that a smooth 
function $u:\Omega\subset\R^n\to\R$ is harmonic (i.e. $\Delta u =0$) in an open set $\Omega$  if and only if it satisfies the mean value property, that is
\begin{equation}
\label{media}
u(x)=\dashint_{B_r(x)} u(y)dy,\qquad \text{whenever  $B_r(x)\subset 	\subset \Omega$}.
\end{equation}
This remarkable characterization of harmonic functions provided a fertile ground for extensive developments and applications. Additionally, such a representation holds in some sense for harmonic functions with respect to more general differential operators. In fact, similar properties can be obtained for quasi-linear operators such as the $p$-Laplace operator $\Delta_p $,
in an asymptotic form. More precisely, a first result is due to Manfredi, Parviainen and Rossi, who proved  in \cite{MPR}  that
if $p\in (1,\infty]$, a continuous function $u:\Omega\to\R$ is $p$-harmonic in $\Omega$ if and only if (in the viscosity sense) 
\begin{equation}
\label{mediav}
u(x)=\frac{2+n}{p+n}\,\dashint_{B_r(x)}u(y)dy+
\frac{p-2}{2p+2n}\Big(\max_{\overline{B_r(x)}}u+\min_{\overline{B_r(x)}}u\Big)
+ o (r^2),
\end{equation}
 as the radius $r$ of the ball vanishes. This characterization also encouraged  a series of new research, such as \cite{harmonious, parharm}, whereas other very nice results were obtained in sequel, see e.g. \cite{tizi,magna,lindgt}. Further new and very interesting  related work is contained in \cite{tesodyn,arroyo2020p}.
 
Notice that formula \eqref {mediav} boils down to \eqref{media} for $p=2,$ up to a rest of order $ o (r^2)$, and that  it holds true in the classical sense at those points $x\in{\Omega}$ such that $u$ is $C^2$ around $x$ and such that the gradient of $u$ does not vanish at $x$.
In the case $p=\infty$ the formula fails in the classical sense, since $|x|^{4/3}-|y|^{4/3}$ is $\infty$-harmonic in $\R^2$ in the viscosity sense, but \eqref{mediav} fails to hold point-wisely. If $p\in (1,\infty)$ and $n=2$ the characterization holds in the classical sense (see \cite{AL,LM}). Finally, the  limiting case $p=1$ was investigated in 2012 in \cite{MPR2}.

Nonlocal operators have been under scrutiny in the past decade. The interest, not only from the purely mathematical point of view, has exponentially risen, and it was natural to ask  the questions affirmatively answered  in the classical case, to the respective fractional counterparts. 

The investigation of  the validity of a mean value property in the {\em nonlocal} linear case, that is for fractional harmonic functions, provided a first positive answer. Let   $s\in (0,1)$, 
we define formally

\[
\Ds{s} u(x) := C(n,s) \lim_{r\to 0} \int_{\Rn \setminus B_r} \frac{u(x)-u(x-y)}{|y|^{n+2s}} \, dy,\qquad 
C(n,s)= 	{\frac{2^{2s} s \Gamma\left(\frac{n}2 +s\right)} {\pi^{\frac{n}2} \Gamma(1-s)}}.
\]
The equivalence between $s$-harmonic functions  (i.e., functions  that satisfy $(-\Delta)^s u=0$) and the fractional mean value property is proved in \cite{AbatLarge} (see also \cite{bucur}), with the fractional mean kernel given by 
\eqlab{ \label{vbg} \M_r^s u(x)=  c(n,s)  r^{2s} \int_{\Rn \setminus B_r} \frac{u(x-y) }{(|y|^2-r^2)^s|y|^n} \, dy.
}
Here,
	\eqlab{ \label{vic1}
	c(n,s) =  \left[\int_{\Rn \setminus B_r} \frac{r^{2s} dy }{(|y|^2-r^2)^s|y|^n} \right]^{-1}= \frac{\Gamma \left(\frac{n}2\right) \sin \pi s } {\pi^{{n/2}+1} }.
	}
	The formula \eqref{vbg} is far from being the outcome of a recent curiosity. It was introduced in 1967 (up to the authors knowledge) in \cite[formula (1.6.2)]{Landkof}, and was recently fleshed out for its connection with the fractional Laplace operator. Different applications rose from such a formula, just to name a few \cite{abatsven,kypri,bucur2020}. We point out furthermore that the formula in \eqref{vbg} is consistent with the classical case, as expected: as $s \to 1^-$, the fractional Laplacian goes to the classical Laplacian, and the mean value kernel goes to the classical mean value on the boundary of the ball (see \cite{Landkof,ClaMarNA}), that is 
	\eqlab{ \label{s1io}\lim_{s \to 1^-} \M_r^s u(x) = \dashint_{B_r(x)} u(y) \, dy.
	}

An asymptotic expansion can be obtained also for fractional anisotropic operators (that include the case of the fractional Laplacian), as one can observe in \cite{ClaMarNA}.  In particular, the result is that a continuous function $u$ is $s$-harmonic in the viscosity sense if and only if \eqref{vbg} holds in a viscosity sense up to  a rest of order two, namely
	\eqlab{
	\label{vbgd} u(x)=  
	 \M_r^s u(x)+o(r^{2}), \qquad  \mbox{ as } \; r \to 0^+.
	}

The goal of this paper is to extend the analysis of the nonlocal case to  the fractional $p$-Laplace operator.  Namely, the fractional (variational) $p$-Laplacian is 
the differential (in a suitable Banach space) of the convex functional 
\[
u\mapsto \frac{1}{p}[u]_{s,p}^p := \frac{1}{p}\iint_{\R^{2n}} \frac{|u(x) - u(y)|^p}{|x - y|^{n+ps}} dx\, dy
\]
and is formally defined as
$$
(-\Delta)^s_pu(x)=\lim_{\eps \to 0^+}\L_\eps^{s,p} u(x),\qquad 
\L_\eps^{s,p} u(x) 
:=\int_{|y|>\eps} \frac{|u(x) -u(x-y)|^{p-2} (u(x)-u(x-y) )}{|y|^{n+sp}}\, dy.
$$
Notice that this  definition is consistent, up to a normalization, with the linear
operator $(-\Delta)^s= (-\Delta)^s_2$. The interested reader can appeal to \cite{hitch,Mingione,brasck,agne,equiv,giampy} to find an extensive theory on the fractional $p$-Laplace operator and other very useful references.

\smallskip

A first issue towards our goal is to identify a reasonable version of a \emph{nonlocal, nolinear, mean value property}.  Up to the authors' knowledge, this is the first attempt to obtain similar properties in the nonlocal, nonlinear, case. Consequently,  on the one hand, the argument is new, so we cannot base our results on any reference.  On the other hand, intuitively one can say that a formula could be reasonable if it were consistent with the already known problems: the nonlocal case of $p=2$ (corresponding to the case of the fractional Laplacian), and the local, nonlinear case, $ s=1$ (corresponding to the case of the classical $p$-Laplacian). 
The main result that we propose is the following.

\medskip
\textbf{Main result 1.}\emph{ Let $p\geq 2$, $\Omega\subset \Rn$ be an open set and let $u\in C(\Omega)\cap L^{\infty}(\Rn)$ be a non constant function. Then 
\[ 
	\frp u(x) =0 
	\]
in the viscosity sense if and only if 
	\bgs{ 
\int_{\Rn \setminus B_r} \left(\frac{|u(x) -u(x-y)|}{|y|^s}\right)^{p-2}  \frac{u(x)-u(x-y)}{|y|^n(|y|^2-r^2)^s}\, dy=
	 	o_r(1) \qquad \mbox{ as } \quad r \to 0^+
		 } 
	holds in the viscosity sense for all $x\in \Omega$.
}
\medskip

This main achievement is precisely stated in Theorem \ref{theorem}. The proof of this main theorem is based on an expansion formula for the fractional $p$-Laplacian for smooth functions, that we do in Theorem \ref{asymp23}, obtaining by means of Taylor expansions and a very careful handling of the remainders, that 
\bgs{
	\int_{\Rn \setminus B_r} \left(\frac{|u(x) -u(x-y)|}{|y|^s}\right)^{p-2}   \frac{u(x)-u(x-y)}{|y|^n(|y|^2-r^2)^s}\, dy= \frp u(x)
			 	  + 
			 		  \mathcal O(r^{2-2s
			 		}) , \qquad \mbox {as $r\to 0^+$.} 
	}
It is enough then to use the viscosity setting to obtain our main result. 		

\medskip
If we try to write the  formula in this first main result in a way that reflects the usual ``$u(x)$ equals its mean value property'', we obtain
\eqlab{ \label{intrpeq1234}
u(x) =  &\; \left[ \int_{\Rn \setminus B_r} \left(\frac{|u(x) -u(x-y)|}{|y|^s}\right)^{p-2}  \frac{ dy}{|y|^n(|y|^2-r^2)^s} \right]^{-1}\\
&\; \left[\int_{\Rn \setminus B_r} \left(\frac{|u(x) -u(x-y)|}{|y|^s}\right)^{p-2}  \frac{u(x-y)}{|y|^n(|y|^2-r^2)^s}\, dy
	 	+ o_r(1) \right]\qquad \mbox{ as } \quad r \to 0^+.
}
 If we now compare  formula \eqref{intrpeq1234} with our desired assumptions, we notice that for $p=2$, supported by \eqref{vic1}, we recover the result in \eqref{vbgd}. Our candidate for the role of the  \emph{nonlocal, nonlinear, mean property} gives indeed the usual formula for the fractional $s$-mean value property.  Looking also at the validity of \eqref{s1io} in our nonlinear case, we come up with an interesting ancillary result as  $s \to 1^-$. In Proposition \ref{mvploc} we obtain an asymptotic expansion for the (classical) variational $p$-Laplace operator, and the equivalence in the viscosity sense between $p$-harmonic and a $p$-mean value property in Theorem \ref{theorem2345}.  The expression obtained by us has some similitudes with other formulas from the literature: compare e.g \eqref{31mvploc} with \cite[Theorem 6.1]{lindgt} that has recently appeared in the literature,  or to \cite[Theorem 1.1]{tizi}, and \eqref{3} with \cite{MPR} (see Remark \ref{mprrmk} for further details). 
 
 Thus, it appears that our formula is consistent with the expressions in the literature for $s=1$, and with the case $p=2$. In our opinion, such a consistency suggests that formula  \eqref{intrpeq1234} is a reasonable proposal  for a \emph{nonlocal, nonlinear, mean value property}.

We mention that a downside of \eqref{intrpeq1234} is that it does not allow to obtain a clean ``$u(x)$ equal to its mean value property'', as customary, given the dependence of $u(x)$ itself, inside the integral, of the right hand side. This ``inconvenient'', nonetheless, does not disappear  in the local setting: the reader can see the already mentioned works \cite{lindgt,tizi}. 

  	\bigskip

 In the second part  of the paper  we investigate a different nonlocal version of the $p$-Laplace and of the infinity Laplace operators,  that arise in tug-of-war games, introduced in \cite{graddep,bjor}.  
 To avoid overloading the notations, we summarize the results on these two nonlocal operators as follows. We denote by $(-\Delta)^s_{p,\pm}$ the nonlocal $p$-Laplace (as in \cite[Section 4]{graddep})
 and by $M_r^{s,p,\pm}$ the ``nonlocal $p$-mean kernel''.   The precise definitions of the operator, of the mean kernel and of viscosity solutions are  given in Subsection \ref{graddepend}.  The asymptotic representation formula is  the content of Theorem~\ref{theorem2}.

\medskip

\textbf{Main result 2.}\emph{ Let $\Omega\subset \Rn$ be an open set and let $u\in C(\Omega)\cap L^{\infty}(\Rn)$. Then 
\[ 
	(-\Delta)_{p,\pm}^s u(x) =0 
	\]
in the viscosity sense if and only if 
	\bgs{
	\lim_{r \to 0^+}  \left(u(x)- M^{s,p,\pm}_r u(x)\right)  = o(r^{2s})  
	}
	holds for all $x\in \Omega$ in the viscosity sense. 	
}
 
 \medskip
 
 An analogous result is stated in Theorem \ref{theorem1} for the infinity Laplacian (introduced in \cite{bjor}) and the ``infinity mean kernel'', defined respectively in Subsection \ref{9876}. The same strategy for the proof as Main Result 1 is adopted for the nonlocal $p$ and infinity Laplacian, by coming up in Theorems \ref{mrspw} and \ref{asympw}, by means of Taylor expansions,  with  formulas which hold for smooth functions, and then passing to the viscosity setting. 
  
 Furthermore,  we study the asymptotic properties of these gradient dependent operators and of the mean kernels as $s \to 1^-$. As a collateral result, we are able to obtain in Proposition \ref{mvploc122} an expansion for the normalized $p$-Laplacian (and the consequence for viscosity solutions in  Theorem \ref{theorem234567}), which up to our knowledge, is new in the literature. In our opinion, the behavior in the limit case $s\to 1^-$ of the formula we propose, justifies here also our choice of the mean value property expression.

We advise the reader interested in mean value formulas for these  two gradient dependent operators here discussed to consult the very recent papers \cite{del2020asymptotic ,lewicka2020non}. Therein, the authors introduce  asymptotic mean value formulas which do not depend on the gradient, making them much useful in applications. 

\bigskip

  To conclude the introduction,  we mention the plan of the paper. The results relative to the fractional (variational) $p$-Laplace operator  are the content of Section \ref{pFract}. Section  \ref{graddepend} contains the results on the nonlocal gradient dependent operators, while in the Appendix \ref{appendicite} we insert some very simple, basic integral asymptotics.
 	
\section{The fractional $p$-Laplacian} \label{pFract}
   
 \subsection{An asymptotic expansion}
   

	Let $p\geq2$. Throughout Section \ref{pFract}, we consider $u$ to be a non constant function.  To simplify the formula in \eqref{intrpeq1234}, we introduce the following notations:
	\eqlab{ \label{fru2}
\D_r^{s,p} u(x): = \int_{|y|>r} \left(\frac{|u(x) -u(x-y)|}{|y|^s}\right)^{p-2}  \frac{dy}{|y|^n(|y|^2-r^2)^s},
	}
	and
	\eqlab{ \label{m1}
	\M^{s,p}_r u(x):=
	&\left( \D_r^{s,p} u(x)\right)^{-1} 
	\int_{|y|>r} \left(\frac{|u(x) -u(x-y)|}{|y|^s}\right)^{p-2}  \frac{u(x-y)}{|y|^n(|y|^2-r^2)^s}\, dy.}
	To make an analogy with the local case,  we may informally say that $\D_r^{s,p} u$ plays the ``nonlocal'' role of $\nabla u(x)$ (see also 
	and Proposition \ref{ghirs}, for the limit as $s \to 1^-$), and $\M^{s,p}_r u$ the role of a $(s,p)$-mean kernel. Both $\D_r^{s,p}$ and $\M_r^{s,p}$ naturally appear when we make an asymptotic expansion for smooth functions, that we do in Theorem \ref{asymp23}.  Notice also that for $p=2$,  $\M_r^{s,2} u$  is given by \eqref{vbgd} (and $\D_r^{s,2} u(x) = c(n,s)^{-1}r^{-2s}$).
	 
	The next proposition motivates working with non-constant functions, and justifies \eqref{m1} as a good definition. 

\begin{remark} \label{posd}Let $u\colon \Rn \to \R$ be such that, for some $x\in \Rn$ there exist $z_x\in \Rn$ and $r_{x}\in (0,|x-z_x|/2)$ such that 
\eqlab{ \label{fru1} u(x) \neq u(z) \quad \forall \; z \in B_{ r_x}(z_x).}
Then there exist some  $c_x>0$ such that, for all $r< r_x$, it holds that $\D_r^{s,p} u(x)\geq c_x$.
\end{remark}

\noindent Notice that if $u$ is a continuous, non-constant function, for some $x\in \Rn$ there exist $z_x\in \Rn$ and $r_{x}\in (0,|x-z_x|/2)$ such that \eqref{fru1} is accomplished. \\
The fact that $\D_r^{s,p} u(x)$ is bounded strictly away from zero is not difficult to see, we prove it however for completeness. 
\begin{proof}
We have that 
	\[ |y|^2-r^2 \leq |y|^2,\] hence
	\[\D_r^{s,p} u(x) \geq \int_{|y|>r} \frac{|u(x) -u(x-y)|^{p-2}}{|y|^ {n+sp}} \, dy \geq \int_{\C B_r(x)}\frac{|u(x) -u(y)|^{p-2}}{|x-y|^ {n+sp}} \, dy ,\]
	 by changing variables.
	For any $r<r_x/2$, $B_{r_x}(z_x) \subset \C B_r(x)$, thus
	\[ \D_r^{s,p}u(x) \geq  \int_{B_{ r_x}(z_x) } \frac{|u(x) -u(y)|^{p-2}}{|x-y|^ {n+sp}} \, dy:=c_x,\]
	with $c_x$ positive, independent of $r$.
	\end{proof}

\begin{remark}
Notice that it is quite natural to assume that $u$ is not constant and it is  similar to what is required in the local case,
namely $\nabla u(x)\neq 0$ (see the proof of \cite[Theorem 2]{MPR}).
	\end{remark}

Using the notations in \eqref{fru2} and \eqref{m1}, we obtain the following asymptotic property for smooth functions.
\begin{theorem}\label{asymp23}
		Let  $\eta>0, x\in\Rn$ and  $u\in C^2(B_\eta(x))\cap L^{\infty}(\Rn)$. Then	
	\eqlab{ \label{expansion} 
		\D_r^{s,p} u(x)\big( u(x)- \M^{s,p}_r u(x) \big) = 
					 \frp u(x)
			 	  + 			 		  \mathcal O(r^{2-2s
			 		}) 
	}
			as $r\to 0$.
\end{theorem}
	
	Substituting our notations, we remark that \eqref{expansion}  is 
	\bgs{
	\int_{\Rn \setminus B_r} \left(\frac{|u(x) -u(x-y)|}{|y|^s}\right)^{p-2}   \frac{u(x)-u(x-y)}{|y|^n(|y|^2-r^2)^s}\, dy= \frp u(x)
			 	  + 
			 		  \mathcal O(r^{2-2s
			 		}) , \qquad \mbox {as $r\to 0$.} 
	}
		
\begin{proof} We note that the constants may change value from line to line. We fix an arbitrary $\bar \eps$ \textcolor{black}{(not necessarily small)}, the corresponding $r:=r(\bar \eps)\in(0,\eta/2)$ as in \eqref{bareps}, and some number $0<\eps<\min\{\bar \eps, r\}$, to be taken arbitrarily small. 

Starting from the definition, we have that
\bgs{
	 & \L^{s,p}_\eps u(x) =  \int_{\eps<|y|<r}  \frac{ |u(x)-u(x-y)|^{p-2} (u(x)-u(x-y) ) }{ |y|^{n+sp} }\, dy
	 \\
	 &\; +  \int_{|y|>r}  \frac{ |u(x)-u(x-y)|^{p-2} (u(x)-u(x-y) ) }{|y|^{n+s(p-2)} } \left( \frac{1}{|y|^{2s}} -\frac{1}{(|y|^2-r^2)^s  } \right) \, dy
	 \\ 
	 &\; + \int_{|y|>r}  \frac{ |u(x)-u(x-y)|^{p-2} (u(x)-u(x-y) ) }{|y|^{n+s(p-2)}(|y|^2-r^2)^s  } \, dy.
}
Thus we obtain that
	\eqlab{
		 \label{epsr} 
		 	&\; \L^{s,p}_\eps u(x) + \int_{|y|>r} 
			\frac{|u(x)-u(x-y)|^{p-2}  {u(x-y)}  }{|y|^{n+s(p-2)}(|y|^2-r^2)^s} dy \\ 
			=&\; u(x) \int_{|y|>r}  \frac{  |u(x)-u(x-y)|^{p-2} }{|y|^{n+s(p-2)}(|y|^2-r^2)^s} dy	
			\\
			&\; + \int_{\eps<|y|<r}  \frac{ |u(x)-u(x-y)|^{p-2} (u(x)-u(x-y) )}{|y|^{n+sp} } dy
			\\
			 &\;+\int_{|y|>r}   \frac{ |u(x)-u(x-y)|^{p-2}   (u(x)-u(x-y)  ) }{|y|^{n+s(p-2)}} \left( \frac{1}{|y|^{2s}} -\frac{1}{(|y|^2-r^2)^s} \right) dy
			  \\
			:=&\;u(x) \int_{|y|>r}  \frac{ |u(x)-u(x-y)|^{p-2} }{|y|^{n+s(p-2)}(|y|^2-r^2)^s} dy + I^s_{\eps}(r)+ J(r).
}
Since $u\in C^2(B_\eta(x))$, \textcolor{black}{we can proceed as in \eqref{reps}, and by employing} \eqref{ireps} and \eqref{jreps}, we get that

	\eqlab{\label{iepsr} 
	\lim_{\eps  \to 0^+} I^s_{\eps}(r)= \mathcal O(r^{p(1-s) }),
		}
		(see also  \cite[Lemma 3.6]{equiv}).
Looking for an estimate on $J(r)$, we split it into two parts 
	\bgs{ 
			J(r)	&=   \int_{|y|>r} \frac{ |u(x)-u(x-y)|^{p-2} ( u(x)-u(x-y) )}{|y|^{n+s(p-2)} }
			\left( \frac{1}{|y|^{2s}} -\frac{1}{(|y|^2-r^2)^s} \right)  dy
			\\
		&= r^{-sp} \int_{|y|>1} \frac{ |u(x)-u(x-ry)|^{p-2}  (u(x)-u(x-ry) )}{|y|^{n+s(p-2)} }  \left( \frac{1}{|y|^{2s}} -\frac{1}{(|y|^2-1)^s} \right) dy
		\\
		&= r^{-sp} \bigg[\int_{|y|>{\frac{\eta}r}} \frac{ |u(x)-u(x-ry)|^{p-2}  (u(x)-u(x-ry) )}{|y|^{n+s(p-2)} } \left( \frac{1}{|y|^{2s}} -\frac{1}{(|y|^2-1)^s} \right)  dy 
		\\
		&+  \int_{1<|y|<{\frac{\eta}r}} \frac{ |u(x)-u(x-ry)|^{p-2}  (u(x)-u(x-ry) ) }{|y|^{n+s(p-2)} } \left( \frac{1}{|y|^{2s}} -\frac{1}{(|y|^2-1)^s} \right)  \bigg] dy \\
		&=:r^{-sp} \left(J_1(r)+J_2(r)\right).
		}
We have that
	\[
	 |u(x)-u(x-y)|^{p-1}\leq c(|u(x)|^{p-1} + |u(x-y)|^{p-1}),
	 \]
thus we obtain the bound
	\bgs{\label{aaff}
		 |J_1(r)| \leq 
		 C \| u\|^{p-1}_{L^\infty(\Rn)}\int_{\frac{\eta}r}^\infty \frac{dt}{t^{sp+1}} 
		 \bigg| 1- \frac{1}{(1-\frac{1}{t^2})^s} \bigg|.
	}
The fact that
	\[ 
		J_1(r) =\mathcal O (r^{2+sp})
	\]
	follows from  \eqref{trois}.
For $J_2$, by symmetry we write
\bgs{
 J_2(r) =&\; \frac12 \int_{1<|y|<{\frac{\eta}r}} \frac{ |u(x)-u(x-ry)|^{p-2}  (u(x)-u(x-ry) )}{|y|^{n+s(p-2)} } \left( \frac{1}{|y|^{2s}} -\frac{1}{(|y|^2-1)^s} \right) dy 
 \\
 &\; +  \frac12 \int_{1<|y|<{\frac{\eta}r}} \frac{  |u(x)-u(x+ry)|^{p-2}  (u(x)-u(x+ry) )}{|y|^{n+s(p-2)} }  \left( \frac{1}{|y|^{2s}} -\frac{1}{(|y|^2-1)^s} \right)dy
 \\
 = &\; \frac12 \int_{1<|y|<{\frac{\eta}r}} \frac{ |u(x)-u(x-ry)|^{p-2}  (2u(x)-u(x-ry) -u(x+ry)) }{|y|^{n+s(p-2)} } \left( \frac{1}{|y|^{2s}} -\frac{1}{(|y|^2-1)^s} \right) dy
 \\
 &\; +  \frac12 \int_{1<|y|<{\frac{\eta}r}} \frac{ \left(|u(x)-u(x+ry)|^{p-2}  -|u(x)-u(x-ry)|^{p-2} \right)  (u(x)-u(x+ry) ) }{|y|^{n+s(p-2)} } \\ 
 &\; \qquad\qquad \left( \frac{1}{|y|^{2s}} -\frac{1}{(|y|^2-1)^s} \right) dy 
.
}
We proceed using \eqref{sumsum1} and \eqref{sumsum2}. For $r$ small enough, we have that 
	\eqlab{\label{strr}
			\left|J_2(R)\right|\leq &\; C r^p \int_1^{\frac{\eta}r}\rho^{p-1-s(p-2)} \left(\frac{1}{(\rho^2-1)^s}- \frac{1}{\rho^{2s}}  \right)\,d\rho
			\\
			\leq &\;C r^{p- (p-2)(1-s)} \int_1^{\frac{\eta}r} \rho  \left(\frac{1}{(\rho^2-1)^s}- \frac{1}{\rho^{2s}}  \right)\,d\rho.
					\\
			\leq &\; C r^{p-(p-2)(1-s)},
	} 
	from \eqref{eun}, with $C$ depending also on $\eta$.
	This yields that $J_2(r)=\mathcal O(r^{2+sp-2s})$.

\medskip 
\noindent 
It follows that
	\[
		J(r)= \mathcal O (r^{2-2s}).
	\]
	Looking back at \eqref{epsr}, using this and recalling \eqref{iepsr}, by sending $\eps \to 0^+$, we obtain that
	\bgs{ \label{fag}
		 \frp u(x) + \int_{|y|>r}
		\frac{ |u(x)-u(x-y)|^{p-2}  {u(x-y)} }{|y|^{n+s(p-2)} (|y|^2-r^2)^s}  dy
			=&u(x) \D_r^{s,p} u(x) 
			+\mathcal O(r^{2-2s}). }
	This concludes the proof of the Theorem.
\end{proof}

\noindent
It is a property of mean value kernels $\M_r u(x)$ that they   converge to $u(x)$ as $r \to 0^+$ both in the 
local (linear and nonlinear) and in the nonlocal linear setting. In our case, due to the presence of $\D_r^{s,p}u,$
we have this property when $\nabla u(x) \neq 0$ only for a limited range of values of $p$, a range depending on $s$ and becoming larger as $s\to 1^-$. For other values of $p$, we were not able to obtain such a result. 
More precisely, we have the following proposition. 

\begin{proposition}\label{boh1}
	Let $\eta>0$ and $x\in \Rn$. If $u\in C^2(B_\eta(x))\cap L^\infty(\Rn)$ is such that $\nabla u(x) \neq 0$, and $s,p$ are such that
	\[
	p \in \left[2,\frac{2}{1-s}\right),
	\]
	it holds that
	\[ 
	\lim_{r  \to 0^+} \M_r^{s,p} u(x) = u(x).
	\]
\end{proposition}

\begin{proof}
	There is some $r\in(0,\eta/4)$ such that $\nabla  u(y) \neq 0$ for all $y\in B_{2 r}(x)$. Then
	\bgs{
		\D_r^{s,p} u(x) \geq &\; \int_{B_{2r} \setminus B_r} \frac{ |u(x)-u(x-y)|^{p-2} }{|y|^{n+s(p-2)} (|y|^2-r^2)^s} \, dy 
		\\
		= &\; \int_{B_{2r} \setminus B_r}  \left|\nabla u(\xi) \cdot \frac{y}{|y|} \right| ^{p-2}  |y|^{(p-2)(1-s) -n} (|y|^2-r^2)^{-s} \, dy ,
	}
	for some $\xi \in B_{2r}(x)$. 
	Therefore, using \eqref{cnp56}
	\bgs{
		\D_r^{s,p} u(x)  \geq C_{p,n} |\nabla u(\xi)|^{p-2} r^{(p-2)(1-s) -2s},
	}
	which for $p$ in the given range, allows to say that
	\[ 
	\lim_{r  \to 0^+} \D_r^{s,p} u(x)= \infty.
	\]
	From Proposition \ref{posd} and Theorem \ref{asymp23}, we obtain
	\bgs{ 	
		\lim_{r  \to 0^+} \left( u(x) - \M_r^{s,p} u(x) \right) 
		= \lim_{r  \to 0^+} 
		\left(\D_r^{s,p} u(x) \right)^{-1} 
		\left( (-\Delta)^s_p u(x)  + \mathcal O (r^{2-2s})  \right),
	}
	and the conclusion is settled.
\end{proof}


\subsection{Viscosity setting}\label{visc}
For the viscosity setting of the $(s,p)$-Laplacian, see the paper \cite{Lindgren} (and also \cite{caffy,giampy,equiv}). 
As a first thing, we recall the definition of viscosity solutions.

\begin{definition} A function $u\in L^\infty(\Rn)$, upper (lower) semi-continuous in $\overline \Omega$ is a viscosity subsolution  (supersolution) in $\Omega$ of 
	\[
		(-\Delta)_p^s u=0, \qquad \mbox{ and we write } \quad (-\Delta)_p^s u \leq \,(\geq) \,  0
	\]
	if for every $x\in\Omega$, any neighborhood $U=U(x)\subset \Omega$ and any $\varphi \in C^2(\overline U)$ such that
	\eqlab{\label{visc23} 
		& \varphi (x)=u(x)
		\\
		& \varphi(y)> \, (<) \,u(y), \quad \mbox{ for any } y\in U\setminus \{x\},
		}
	if we let
		\syslab[v =]{&\varphi , \quad \mbox{ in }  U	\\
							&u, \quad \mbox{ in } \Rn\setminus U	
							\label{vvv234} 
							}
				then
				 \bgs{
				 	(-\Delta)_p^s v (x) \leq \,(\geq)\,0.
				 }			
	A viscosity solution of $(-\Delta)_p^s u =0$ is a (continuous) function that is both a subsolution and a supersolution. 
	\end{definition}	
	
	We define here what we mean when we say that an asymptotic expansion holds in the viscosity sense. 
	\begin{definition} \label{meanvisc23}Let $u \in L^\infty(\Rn)$ be upper (lower) semi-continuous in $\Omega$. We say that 
\bgs{
\label{fggd}
	\textcolor{black}{	\D^{s,p}_r u(x)  \big( u(x)- \M^{s,p}_r u(x)\big)  =  o_r(1)}
	 \qquad \mbox{ as } \quad r  \to 0^+
	}
	holds in the viscosity sense
if for any 
	neighborhood $U=U(x)\subset \Omega$ and any $\varphi \in C^2(\overline U)$ such that \eqref{visc23} holds, and
if we let $v$ be defined as in \eqref{vvv234}, then both
\bgs{
\label{fggd}
	\liminf_{r \to 0^+} \D^{s,p}_r u(x)  \big( u(x)- \M^{s,p}_r u(x)\big)
	  \geq 0 
	}
	and
	\bgs{
\label{fggd}
	\limsup_{r \to 0^+}  \D^{s,p}_r u(x)  \big( u(x)- \M^{s,p}_r u(x)\big) \leq 0 
	}
	hold point wisely.
\end{definition}	

The result for viscosity solutions is a consequence of the asymptotic expansion for smooth functions, and goes as follows.
\begin{theorem}\label{theorem} Let $\Omega\subset \Rn$ be an open set and let $u\in C(\Omega)\cap L^{\infty}(\Rn)$. Then 
\[ 
	\frp u(x) =0 
	\]
in the viscosity sense if and only if 
	\eqlab{ \label{peq123}
	 			 		\D^{s,p}_r u(x)  \big( u(x)- \M^{s,p}_r u(x)\big)=
	 			 		o_r(1)
	 \qquad \mbox{ as } \quad r  \to 0^+
			 } 
	holds for all $x\in \Omega$ in the viscosity sense. 	
\end{theorem}
 
\begin{proof}
For $x\in \Omega$ and any $U(x)$ neighborhood of $x$, defining 
		$v$ as in \eqref{vvv234}, we have that $v\in C^2(U(x))\cap L^\infty(\Rn)$.	
		By Theorem \ref{asymp23} we have that
		\eqlab{
			\label{asaf23} 
				\D_r^{s,p}v(x)   \big (v(x)- \M_r^{s,p} v(x) \big) =(-\Delta)^s_p v(x) + \mathcal O(r^{2-2s}),
		}
which allows to obtain the conclusion.

\end{proof}

\subsection{Asymptotics as $s\to 1^-$}\label{asymp-sec}

\vskip2pt
\noindent
 	We prove here that sending $s\to 1^-$, for a smooth enough function the fractional $p$-Laplace operator approaches the $p$-Laplacian, defined as
 	\eqlab{\label{locpl}
 		- \Delta_p u= \mbox{div} (|\nabla u|^{p-2} \nabla u)
 		.}
 	The result is known in the mathematical community, see \cite{IshiiNakamura}. We note for the interested reader that the limit case $s \to 1^-$ for fractional problems has an extensive history, see, e.g. \cite{BBM,ponce}.

 	We give here a complete proof of this result, on the one hand for the reader convenience and 
 	on the other hand since some estimates here introduced are heavily
 	used throughout Section 2.
 
\begin{theorem}\label{limits1}
Let $\Omega\subset \Rn$ be an open set and let  $u\in C^{2}(\Omega) \cap L^\infty(\Rn)$. Then  
\bgs{
		 \lim_{s\to 1^-} (1-s) (-\Delta)_p^s u(x) =-C_{p,n} \, \Delta_p u(x),}
		 where $C_{p,n}>0$		 
	for every $x\in \Omega$ such that $\nabla u(x) \neq 0$.
	\end{theorem}
	\begin{proof}
	Since $u\in C^2(\Omega)$, for any $x\in \Omega$ we have that for any $\bar \eps>0$ there exists $ r= r(\bar \eps)>0$ such that 
	\eqlab{ 
	\label{bareps} 
	\mbox{ for any } y\in B_r(x) \subset \Omega, \qquad |D^2u(x)-D^2u(x+y)| < \bar \eps.
	} 
	We fix an arbitrary $\bar \eps$ (as small as we wish), the corresponding $r$ and some number $0<\eps<\min\{\bar \eps, r\}$, to be taken arbitrarily small.
	\\
	\noindent We notice that
	\eqlab{ \label{jkk} (-\Delta)_p^s u(x)= \lim_{\eps \to 0} \L^{s,p}_\eps u(x) = \L^{s,p}_r u(x) +  \lim_{\eps \to 0} \left( \L^{s,p}_\eps u(x)  - \L^{s,p}_r u(x)\right).
	}
	For the first term in this sum, we have that
	\bgs{
	\L^{s,p}_r u(x) = &\;
		\int_{|y|>r}\frac{ |u(x)-u(x-y)|^{p-2} (u(x)-u(x-y))} {|y|^{n+sp}}\, dy
	\leq 2^{p-1} \|u\|_{L^\infty(\Rn)}^{p-1} \omega_n \int_r^\infty \rho^{-1-sp}\, d\rho
	\\
	=&\; C(n,p,\|u\|_{L^\infty(\Rn)}) \frac{r^{-sp}}{sp}.
	}
	Notice that
	\eqlab{
		\label{pnonloc}
				\lim_{s \to 1^-} (1-s)\L^{s,p}_r u(x) = 0.
				}
	Now by symmetry
	\eqlab{ \label{reps}
			&\;2\Big( \L_\eps^{s,p} u(x)  -\L_r^{s,p} u(x) \Big)= 
			2\int_{B_r\setminus B_\eps} \frac{ |u(x)-u(x-y)|^{p-2} (u(x)-u(x-y))} {|y|^{n+sp}}\, dy
			\\
			=&\; \int_{B_r\setminus B_\eps} \frac{ |u(x)-u(x-y)|^{p-2} (u(x)-u(x-y))} {|y|^{n+sp}}\, dy 
			\\
			&\;+ \int_{B_r\setminus B_\eps} \frac{ |u(x)-u(x+y)|^{p-2} (u(x)-u(x+y))} {|y|^{n+sp}}\, dy 
			\\
			= &\; \int_{B_r\setminus B_\eps}  \frac{ |u(x)-u(x-y)|^{p-2} (2u(x)-u(x-y) -u(x+y))} {|y|^{n+sp}}\, dy 
			\\
			&\;+\int_{B_r\setminus B_\eps} \frac{ \Big(|u(x)-u(x+y)|^{p-2} - |u(x)-u(x-y)|^{p-2}\Big) (u(x)-u(x+y))} {|y|^{n+sp}}\, dy
			\\
			=:&\; I_{r,\eps} (x) + J_{r,\eps} (x). 
						}
			Using a Taylor expansion, there exist $\underline \delta, \overline \delta   \in (0,1)$ such that 
			\[
				u(x)-u(x-y) =  \nabla u(x) \cdot y -\frac12 \langle D^2 u(x- \underline \delta y)y,y\rangle ,\qquad 	u(x)-u(x+y) = - \nabla u(x) \cdot y -\frac12 \langle D^2 u(x+\overline\delta  y)y,y\rangle.
				\]
			Having that $|\underline \delta y|, |\overline \delta y|\leq |y|<r$, recalling \eqref{bareps}, we get that $\left|\langle (D^2 u(x)-D^2 u(x- \underline \delta y))y,y\rangle\right| \leq \bar \eps |y|^2$, hence
				\eqlab{\label{11}
				 2u&\;(x)-u(x-y)-u(x+y) =-\langle D^2u(x) y,y\rangle  \\
				 &\;+ \frac12 (\langle D^2u(x) y,y\rangle -\langle D^2u(x- \underline \delta y) y,y\rangle)
				  + \frac12 (\langle D^2u(x) y,y\rangle  -\langle D^2u(x+ \overline \delta y) y,y\rangle)
				 \\
				 =&\; -\langle D^2u(x) y,y\rangle  + T_1,\qquad	\qquad\qquad\qquad  \mbox{ with } \; \;  |T_1|\leq \bar \eps |y|^2.
			} 			
	{Also denoting $\omega ={y}/{|y|} \in \Sf$ and taking the Taylor expansion for the function $f(x)=|a-x b|^{p-2}$, we obtain
		\eqlab{\label{ff123} |u(x)-u(x-y)|^{p-2}  =&\;
			|y|^{p-2} \big|\nabla  u(x) \cdot \omega	- \frac{|y|}2\langle D^2 u(x- \underline \delta y)\omega, \omega\rangle \big|^{p-2} 
			\\
			= &\; |y|^{p-2} \big|\nabla  u(x) \cdot \omega	
			- \frac{|y|}2
			\left( \langle D^2 u(x )\omega, \omega\rangle + \langle\left( D^2 u(x- \underline \delta y)-D^2 u(x )\right)\omega, \omega\rangle\right
			)\big|^{p-2}  
			\\=&\; |y|^{p-2} \big|\nabla  u(x) \cdot \omega	- \frac{|y|}2\left( \langle D^2 u(x )\omega, \omega\rangle + \mathcal O(\bar \eps)\right) \big|^{p-2} 
			\\			=&\;				|y|^{p-2} | \nabla u(x)\cdot \omega|^{p-2}+  T_2, 
										  \qquad\qquad\qquad \mbox{ with }\; \; |T_2|\leq C |y|^{p-1}.}
			}
	Thus we have that
			\eqlab{\label{sumsum1}
						 |u(x)-u(x-y)|^{p-2} (2u(x)-u(x-y)-u(x+y) ) = &\;
						- |y|^p  | \nabla u(x)\cdot \omega|^{p-2}\langle D^2u(x) \omega ,\omega \rangle \\
						&\; + T_1  |y|^{p-2}  | \nabla u(x)\cdot \omega|^{p-2} + T_3, 
						\\
							&\;  \qquad	 \mbox{ with } \; \;   |T_3|\leq C  |y|^{p+1}.						
					}
						Using this and passing to hyper-spherical coordinates, with the notations in \eqref{reps} we have that
		\eqlab{ \label{ireps}
			I_{r,\eps}(x) =&\; -\int_{\eps}^r\rho^{p-1-sp}  \, d\rho 
		\int_{\Sf}  |\nabla u(x)\cdot \omega|^{p-2} \langle D^2u(x)\omega,\omega\rangle\,   d\omega
		+ I^1_{r,\eps}(x)+I^2_{r,\eps}(x)
		\\
		= &\;- \frac{r^{p(1-s)}-\eps^{p(1-s)} }{p(1-s)} \int_{\Sf} |\nabla u(x)\cdot \omega|^{p-2} \langle D^2u(x)\omega,\omega\rangle \, d\omega  +I^1_{r,\eps}(x)+I^2_{r,\eps}(x)	.
			}
			With the above notations, we have that
			\[  \left|I^1_{r,\eps}(x)\right| \leq  \bar \eps \, C \frac{r^{p(1-s)}-\eps^{p(1-s)} }{p(1-s)} 
			\] 
						and that
			\[ \left|I^2_{r,\eps} (x)\right| \leq C \int_{B_r\setminus B_\eps} |y|^{p-sp}\, dy = C \,  \frac{r^{p(1-s)+1} -\eps^{p(1-s)+1}}{p(1-s) +1} 
			.\]
				This means that
				\[ \lim_{s\to 1^-} \lim_{\eps \to 0^+}(1-s)  \left( I^1_{r,\eps}(x)+  I^2_{r,\eps}(x) \right)=\mathcal O( \bar \eps) .\]
				Thus we get
				\[
				\lim_{s\to 1^-} \lim_{\eps \to 0^+}(1-s)  I_{r,\eps}(x) = -\frac1p\int_{\Sf} |\nabla u(x)\cdot \omega|^{p-2} \langle D^2u(x)\omega,\omega\rangle \, d\omega +\mathcal O( \bar \eps).
				\]
Using again that  $\left|\langle (D^2 u(x)-D^2 u(x- \underline \delta y))y,y\rangle\right| \leq   \bar \eps |y|^2$, we also have that
	\bgs{
				u(x)-u(x-y) = &\; \nabla u(x) \cdot y- \frac12 \langle D^2 u(x)y,y\rangle+\frac12\left(\langle D^2 u(x)y,y\rangle-\langle D^2 u(x- \underline \delta y)y,y\rangle\right) 
			\\ 
			= &\; \nabla u(x) \cdot y- \frac12 \langle D^2 u(x)y,y\rangle+\mathcal O(\bar \eps) |y|^2 
			\\
										 	u(x)-u(x+y) =&\;   - \nabla u(x) \cdot y- \frac12 \langle D^2 u(x)y,y\rangle+\frac12\left(\langle D^2 u(x)y,y\rangle-\langle D^2 u(x+\overline \delta y)y,y\rangle\right)
							 	\\
							 	=& \;- \nabla u(x) \cdot y- \frac12 \langle D^2 u(x)y,y\rangle+\mathcal O(\bar \eps) |y|^2 .
				}
				Taking the second order expansion (i.e, taking the following order of the expansion in \eqref{ff123},  with second order  remainder)	we obtain 
				\bgs{
		&\; |u(x)-u(x+y)|^{p-2}-|u(x)-u(x-y)|^{p-2} \\
		=&\;    |y|^{p-1} (p-2)  (\nabla u(x)\cdot \omega)  
		|\nabla u(x)\cdot \omega)|^{p-4} 
		\left(\langle D^2u(x)\omega, \omega\rangle + \mathcal O (\bar \eps)\right)  + T_4, \qquad |T_4| \leq C|y|^{p}.
	}
	Thus
	\eqlab{ \label{sumsum2}
			&(|u(x)-u(x+y)|^{p-2}-|u(x)-u(x-y)|^{p-2} ) (u(x)-u(x+y)) \\
			&\; =  -   |y|^{p} (p-2) 
		|\nabla u(x)\cdot \omega)|^{p-2} 
		\left(\langle D^2u(x)\omega, \omega\rangle + \mathcal O (\bar \eps)\right)  + T_5,	\qquad\qquad\qquad\;\quad |T_5| \leq C|y|^{p+1} .
		}
		Therefore, with the notation in \eqref{reps}, we get that
	\eqlab{ \label{jreps}
	J_{r,\eps}(x) = -(p-2)\frac{r^{p(1-s)}-\eps^{p(1-s)}}{p(1-s)}
	 \left( \int_{\Sf} \
		|\nabla u(x)\cdot \omega)|^{p-2} \langle D^2u(x)\omega, \omega\rangle  d\omega + \mathcal O(\bar \eps)  \right)
		+ \tilde J_{r,\eps}(x),
		}
		and 
		\bgs{ \left|\tilde J_{r,\eps}(x)\right| \leq  	C \frac{r^{p(1-s)+1}-\eps^{p(1-s)+1}}{p(1-s)+1}.
		}
		We obtain that
		\[ \lim_{s\to 1^-} \lim_{\eps  \to 0^+} J_{r,\eps}(x)= -\frac{p-2}p \int_{\Sf} \
		|\nabla u(x)\cdot \omega)|^{p-2}  \langle D^2u(x)\omega, \omega\rangle\, d\omega +\mathcal O(\bar \eps).
		\]
		\textcolor{black}{Summing the limits for $I_{r,\eps}(x)$ and $J_{r,\eps}(x)$ we obtain
		\[
			\lim_{s\to 1^-} \lim_{\eps  \to 0^+}  \left( \L^{s,p}_\eps u(x)  - \L^{s,p}_r u(x)\right) =  -\frac{p-1}{2p} \int_{\Sf} \
		|\nabla u(x)\cdot \omega)|^{p-2}  \langle D^2u(x)\omega, \omega\rangle\, d\omega +\mathcal O(\bar \eps).
		\]}
		Using this and \eqref{pnonloc} into \eqref{jkk}, it follows that
		\bgs{
		 \lim_{s\to 1^-} (1-s) (-\Delta)_p^s u(x)=&\; \lim_{s\to 1^-} (1-s) \left( \L^{s,p}_r u(x) +  \lim_{\eps  \to 0^+} \left( \L^{s,p}_\eps u(x)  - \L^{s,p}_r u(x)\right) \right) \\
		 =&\;   -\frac{p-1}{2p} \int_{\Sf} \
		|\nabla u(x)\cdot \omega)|^{p-2}  \langle D^2u(x)\omega, \omega\rangle\, d\omega +\mathcal O(\bar \eps).
		}
		Sending $\bar \eps $ to zero, we get that
		\bgs{
		 \lim_{s\to 1^-} (1-s) (-\Delta)_p^s u(x) =-\frac{p-1}{2p} |\nabla u(x)|^{p-2} \int_{\Sf} 
		|z(x) \cdot \omega|^{p-2}  \langle D^2u(x)\omega, \omega\rangle\, d\omega,}
		with $z(x)=\nabla u(x)/|\nabla u(x)| $. 
We follow here the ideas in \cite{IshiiNakamura}. Let $U(x)\in \mathbb M^{n\times n}(\R)$ be an orthogonal matrix, such that $ z(x)= U(x) e_n$, where $e_k$ denotes the $k^{th}$ vector of the canonical basis of $\R^n$. Changing coordinates $\omega'=U(x)\omega$ we obtain
\bgs{ \mathcal I:= &\;
\int_{\Sf} |z(x) \cdot \omega'|^{p-2}  \langle D^2u(x)\omega', \omega'\rangle\, d\omega
= \int_{\Sf} |e_n \cdot \omega |^{p-2}  \langle U(x)^{-1}  D^2u(x)U(x) \omega, \omega\rangle\, d\omega	
\\
=&\; \int_{\Sf} |\omega_n |^{p-2}  \langle B(x) \omega, \omega\rangle\, d\omega
}
where $B(x)=U(x)^{-1}  D^2u(x)U(x)  \in \mathbb M^{n\times n}(\R)$. Then we get that
\bgs{\mathcal I= &\;
 \sum_{i,j=1}^n b_{ij}(x) \int_{\Sf} |\omega_n|^{p-2} \omega_i \omega_j d\omega	= \sum_{j=1}^n b_{jj}(x) \int_{\Sf} |\omega_n|^{p-2} \omega_j^2\, d\omega
		}
		by symmetry. 
		 Now
		\sys[ \int_{\Sf} |\omega_n|^{p-2} \omega_j^2 \, d\omega  =  ]
		{ &\gamma_p ,&& \mbox{ if   } j\neq n \\
			&\gamma_p', && \mbox{ if   }  j=n
					}
					with $\gamma_p, \gamma'_p$ two constants{\footnote{Precisely (see Lemma 2.1 in \cite{IshiiNakamura})
					\[ \gamma_p = \frac{2\Gamma\left(\frac{p-1}2\right) \Gamma\left(\frac12\right)^{n-2} \Gamma\left(\frac{3}2\right) }{\Gamma\left(\frac{p+n}2\right)},
					 \qquad  \gamma'_p =\frac{2\Gamma\left(\frac{p+1}2\right) \Gamma\left(\frac12\right)^{n-1}   }{ \Gamma\left(\frac{p+n}2\right) }.\]
					}}
					  for which $\gamma_p'/\gamma_p=p-1$,
					so
	\bgs{\mathcal I= &\;
=\gamma_p \sum_{j=1}^n b_{jj}(x) + (\gamma_p'-\gamma_p) b_{nn}(x)
= \gamma_p \left( \sum_{j=1}^n b_{jj}(x)  + (p-2) b_{nn}(x)\right)
.
		}	
		We notice that, since $U(x)$ is orthogonal and $D^2 u(x)$ is symmetric,
		\[ 	\sum_{j=1}^n b_{jj}(x)  = \mbox{Tr}B(x)=\mbox{Tr}( U(x)^{-1}  D^2u(x)U(x) )= \mbox{Tr}(D^2 u(x)) = \Delta u(x) \]
		and 
		\bgs{ b_{nn}(x)= &\; \langle U(x)^{-1}  D^2u(x)U(x) e_n, e_n \rangle =\langle  D^2u(x)U(x) e_n, U(x) e_n \rangle=\langle  D^2u(x)z(x), z(x) \rangle  \\
		=&\;|\nabla u|^{-2} \langle  D^2u(x)\nabla u(x) , \nabla u(x)  \rangle = \Delta_\infty u(x) .}
		Therefore 
		\eqlab{ \label{4} \mathcal I =  \gamma_p  \left( \Delta u(x) + (p-2) \Delta_\infty u(x) \right),}
		and this leads to
	\bgs{
		 \lim_{s\to 1^-} (1-s) (-\Delta)_p^s u(x) =-\frac{\gamma_p(p-1)}{2p} |\nabla u(x)|^{p-2}  ( \Delta u(x) + (p-2) \Delta_\infty u(x) ).}
		 Recalling \eqref{locpl} and that 
		 \bgs{
				  \Delta_p u(x)= |\nabla u(x)|^{p-2}  \left( \Delta u(x) + (p-2) \Delta_\infty u(x) \right)
		  }
		 we conclude the proof of the Lemma. 					
	\end{proof}

Next we study the asymptotic behavior of $\M_r^{s,p}$ as $s\to 1^-$.
	
\begin{lemma} \label{ghirs} Let $u\in C^1(\Omega)\cap L^\infty(\Rn)$, 
denoting
\eqlab{ \label{popo} \M_r^p u(x):= \int_{\Sf} |u(x)-u(x-r\omega)|^{p-2} u(x-r\omega) \, d\omega \left(\int_{\Sf} |u(x)-u(x-r\omega)|^{p-2} \, d\omega \right)^{-1} 
}
it holds that
\eqlab{\label{2} \lim_{s\to 1^-} \M_r^{s,p} u(x) = \M_r^p u(x)}
and that
\eqlab{ \label{1} \lim_{s\to 1^-} (1-s) \D_r^{s,p} u(x) =\textcolor{black}{\frac{1}{2r^p}} \int_{\Sf} |u(x)-u(x-r\omega)|^{p-2} \, d\omega ,}
for all $x\in \Omega$, $r>0$ such that $B_{2r}(x) \subset \Omega$. 
\end{lemma}

\begin{proof}
Let $\eps \in (0,1/2)$, to be taken arbitrarily small in the sequel.
We have that
\bgs{
		\D_r^{s,p} u(x)
		=&\;   \int_{|y|>(1+\eps) r} \frac{ |u(x)-u(x-y)|^{p-2} }{|y|^{n+sp-2s}(|y|^2-r^2)^s  } \, dy + \int_{r<|y|<(1+\eps)r} \frac{ |u(x)-u(x-y)|^{p-2} }{|y|^{n+sp-2s}(|y|^2-r^2)^s  } \, dy 
		\\
		=&\; I^{s,\eps}_1+I^{s,\eps}_2.}
		Given that for $|y|>r(1+\eps)$ one has that $|y|^2-r^2 \geq \eps(\eps+2)(1+\eps)^{-2} |y|^2$, we get
		\eqlab{ \label{ghf} 
		|I^{s,\eps}_1| \leq &\;  \frac{(1+\eps^2)^s } {\eps^s(\eps+ 2)^s} c_{n,p,\|u\|_{L^{\infty}(\Rn)}} \int_{(1+\eps)r}^\infty \rho^{-1-sp}\, d\rho 
		\\
		=&\;  \frac{(1+\eps^2)^s } {\eps^s(\eps+ 2)^s}  c_{n,p, \|u\|_{L^{\infty}(\Rn)} } \frac{ [(1+\eps)r]^{-sp}}{s} }
		and it follows that
		\[ \lim_{s\to 1^-} (1-s) I^{s,\eps}_1 =0.\] 
		On the other hand, integrating by parts we have that
		\bgs{
			 I^{s,\eps}_2=&\; \int_{\Sf}  \bigg(\int_r^{(1+\eps)r}  \frac{|u(x)-u(x-\rho \omega )|^{p-2}}{  
			 								\rho^{1 +sp-2s} (\rho^2-r^2)^{s} }\, d\rho \bigg)   d\omega 
			 	\\
			 	=&\; \int_{\Sf}  \Bigg[ \frac{(\rho-r)^{1-s} }{1-s}  \frac{|u(x)-u(x-\rho \omega )|^{p-2}  }{
			 								\rho^{1 +sp-2s}  (\rho+r)^{s} } \Bigg|_r^{(1+\eps)r} \,    d\omega
			 								\\ 
			 								&\;-\int_r^{(1+\eps)r} \frac{(\rho-r)^{1-s} }{1-s}  \frac{d}{d\rho} \left(\frac{|u(x)-u(x-\rho \omega )|^{p-2}}{\rho^{1+sp-2s} (\rho+r)^{s}}  \right) \, d\rho  \Bigg]
			 	\\
			 	=&\; \int_{\Sf} \Bigg[ \frac{(\eps r)^{1-s} }{1-s} \frac{ |u(x)-u(x-(1+\eps)r \omega )|^{p-2}  }{
			 								[(1+\eps)r]^{1 +sp-2s}  [(2+\eps)r]^{s}} \,   d\omega
			 								\\
			 								&\;-\int_r^{(1+\eps)r} \frac{(\rho-r)^{1-s} }{1-s}  \frac{d}{d\rho} \left(\frac{|u(x)-u(x-\rho \omega )|^{p-2}  }{
			 								\rho^{1 +sp-2s}  (\rho+r)^{s}}  \right) \, d\rho  \Bigg].
			 	\\
		} 
	\textcolor{black}{Notice that 
	\bgs{
	\bigg| \int_r^{(1+\eps)r} \frac{(\rho-r)^{1-s} }{1-s}  &\frac{d}{d\rho} \left(\frac{|u(x)-u(x-\rho \omega )|^{p-2} }{ 
			 								\rho^{1 +sp-2s}  (\rho+r)^{s}}  \right) \, d\rho\bigg| 
			 							\leq 
			 							 C \max\{r^{-sp},r^{1-sp}\} \frac{\eps^{2-s}}{1-s},
			 								}}
		hence
		\bgs{ \lim_{s\to 1^-} (1-s)\int_{\Sf}    &  \left[\int_r^{(1+\eps)r} \frac{(\rho-r)^{1-s} }{1-s}  \frac{d}{d\rho} \left(\frac{|u(x)-u(x-\rho \omega )|^{p-2}  }{
			 								\rho^{1 +sp-2s}  (\rho+r)^{s}}  \right) \, d\rho \right]d\omega
			 								= \mathcal O (\eps).
			 								}
	Moreover	
	\bgs{	 &	\lim_{s\to 1^-} (1-s) \int_{\Sf}  \frac{(\eps r)^{1-s} }{1-s} \frac{ |u(x)-u(x-(1+\eps)r \omega )|^{p-2} }{ 
			 								[(1+\eps)r]^{1+sp-2s}  [(2+\eps)r]^{s}}  d\omega  	
			 								\\ 
			 								 =&\;
				\frac{1}{(1+\eps)^{p-1} (2+\eps) r^{p} }	\int_{\Sf}    |u(x)-u(x-(1+\eps)r \omega )|^{p-2}  d\omega .}
Finally,  
\[ \lim_{\eps  \to 0^+} \lim_{s\to 1^-}  (1-s)  I^{s,\eps}_2 =\frac{1}{\textcolor{black}{2}r^p} \int_{\Sf} |u(x)-u(x-r \omega )|^{p-2} \,   d\omega \] and one gets \eqref{1}.
In exactly the same fashion, one proves that 
\bgs{
 \lim_{s\to 1^-}  (1-s)  & \int_{|y|>r} \left(\frac{|u(x) -u(x-y)|}{|y|^s}\right)^{p-2}  \frac{u(x-y)}{|y|^n(|y|^2-r^2)^s}\, dy 
 \\
 =&\;   \frac{1}{\textcolor{black}{2}r^p} \int_{\Sf} |u(x)-u(x-r \omega )|^{p-2} u(x-r\omega) \,   d\omega 
 }	
and \eqref{2} can be concluded.
		\end{proof}
		 
\textcolor{black}{Using the notations from Lemma \ref{ghirs}, we obtain some equivalent asymptotic expansions for the (classical) $p$-Laplacian.
		\begin{proposition} \label{mvploc}
Let $u\in C^2(\Omega)$, then the following equivalent expansions hold
		\begin{subequations} 
		\eqlab{\int_{\partial B_r} |u(x)-u(x-y)|^{p-2} \left(u(x)-u(x-y)\right) d \mathcal H^{n-1}(y)=- c_{n,p} r^p  \Delta_p u(x) +  o(r^{p}),  \label{31mvploc} }
		 \eqlab{  \Big(|\nabla u|^{p-2} + \mathcal O (r) \Big)\Big(u(x) -\M_r^p u(x) \Big) =- c_{n,p} r^2  \Delta_p u(x) +  o(r^{2}) \label{3}}
		\end{subequations}
for all $x\in \Omega$, $r>0$ such that $B_{2r}(x) \subset \Omega$. 
\end{proposition}}
\begin{proof}
We use \eqref{sumsum1}, \eqref{sumsum2} and \eqref{4} with a Peano remainder, to obtain that
	\bgs{
		 & \int_{\Sf}  |u(x)-u(x-r\omega)|^{p-2}  \left(u(x)-u(x-r\omega)\right) d\omega
		 \\
		= &\; -\frac{p-1}{2p} r^p \int_{\Sf} |\nabla u(x)\cdot \omega|^{p-2}\langle D^2u(x)\omega, \omega\rangle d \omega + \textcolor{black}{ o(r^{p})}
		\\
		= &\; -r^p  \frac{\gamma_p(p-1)}{2p} |\nabla u(x)|^{p-2}  ( \Delta u(x) + (p-2) \Delta_\infty u(x) ) + \textcolor{black}{o(r^{p})} \\
		= &\; -r^p  \frac{\gamma_p(p-1)}{2p}  \Delta_p u(x) +  \textcolor{black}{ o(r^{p})} .
	}  
	From this, \eqref{31mvploc} immediately follows with a change of variables. Thus, using the notations in the \eqref{popo},
	\[
		 \left( \int_{\Sf}  |u(x)-u(x-r\omega)|^{p-2} \, d\omega \right)   \left(	u(x) -\M_r^p u(x) \right) =  - \frac{\gamma_p(p-1)}{2p}   \,r^p    \Delta_p u(x) + \textcolor{black}{ o(r^{p})} .
	\]
	Proving in the same way by \eqref{ff123}
that
		\[\int_{\Sf} |u(x)-u(x-r\omega)|^{p-2} \, d\omega = r^{p-2} \int_{\Sf}  |\nabla u(x)\cdot \omega|^{p-2} \, d\omega + \textcolor{black}{ o(r^{p-2})},\]
		and recalling that
		\eqlab{ \label{cnp56}
		\int_{\Sf}  |\nabla u(x)\cdot \omega|^{p-2} \, d\omega=C_{n,p}|\nabla u(x)|^{p-2},
		}
		we obtain
		\[ \Big(|\nabla u(x)|^{p-2} +  \textcolor{black}{o_r(1)} \Big)\Big(u(x) -\M_r^p u(x) \Big) = -\tilde c_{n,p}r^2 \Delta_p u(x) +\textcolor{black}{ o(r^{2})},\] 
with
\[ \tilde c_{n,p}=\frac{\gamma_p (p-1)}{p C_{n,p}}=\frac{(p-1)(p-3)}{2p(p+n-2)}.\]
This concludes the proof of the proposition.
\end{proof}

 \begin{remark} \label{mprrmk} We compare our result in the local setting to the existing literature, pointing out that our expansion is obtained for $p\geq 2$. The formula \eqref{31mvploc} is indeed the same as Theorem 6.1 in the recent paper \cite{lindgt} (we point out that therein the case $p\in (1,2)$ is also studied). Furthermore, the expansion \eqref{31mvploc} has some similitudes to \cite[Theorem 1.1]{tizi} (where, instead, the so-called normalized $p$-Laplacian 
 \eqlab{ \label{renpl}
 		\Delta^{\mathcal N}_p u= \Delta u +(p-2) \Delta_\infty u
 }
 is used).  Indeed, the expansion therein obtained for $n=2$, which we re-write for our purposes (compare the normalized $p$-Laplacian with \eqref{locpl}), says that
 \bgs{
 	 \int_{B_r} |\nabla u(x)|^{p-2} \left( u(x)-u(x-y)\right) dy =&\; \int_{B_r} \left| \frac{ \nabla u(x)}{|\nabla u(x)|} 	\cdot y\right|^{p-2} dy\left( -c_{p} r^2 \Delta_p u(x)   + o(r^2) \right),
 	}
 	so, rescaling,
 	 \bgs{
 	 |\nabla u(x)|^{p-2}   \int_{B_1}\left( u(x)-u(x-ry)\right) dy  
 	 =&\;   
 	 	 -c_{p} r^p \Delta_p u(x)   + o(r^p) 
  	,
}
where the last line holds up to renaming the constant.
  On the other hand,  our expansion differs from the one given in \cite{MPR}, again given for the normalized $p$-Laplacian. Re-written for the $p$-Laplace as in \eqref{locpl},  the very nice formula in \cite{MPR} gives that
 \[ |\nabla u|^{p-2} \left( u(x) - \tilde {\mathcal M}_p u(x) \right) =-  \bar c_{p,n} r^2 \Delta_p u(x) +  o(r^2),\]
 with 
 \[ \tilde {\mathcal M}_p u(x) = \frac{2+n}{p+n} \dashint_{B_r(x)} u(y)\, dy +\frac{p-2}{2(p+n)}  \left(\max_{\overline {B_r}(x)} u(y) + \min_{\overline {B_r}(x)} u(y) \right)\]
 and
 \[ \bar c_{p,n}= \frac{1}{2(p+n)}.\]
The statement \eqref{3}, even though it appears weaker, still allows us to conclude that in the viscosity sense, at points $x\in \Rn$ for which the test functions $v(x)$ satisfy $\nabla v(x) \neq 0$, if $u$ satisfies the mean value property, then $\Delta_p u(x)=0$.
 \end{remark}
 \textcolor{black}{
 More precisely, we state the result for viscosity solutions (which follows from the asymptotic expansion for smooth functions). 
\begin{theorem}\label{theorem2345} Let $\Omega\subset \Rn$ be an open set and let $u\in C(\Omega)\cap L^{\infty}(\Rn)$. Then 
\[ 
	(-\Delta)_p u(x) =0 
	\]
in the viscosity sense if and only if 
	\bgs{ \label{peq123}
	 			 		\lim_{r \to 0^+} \big(u(x)- M^{p}_r u(x)\big)=
	 			 		o(r^p)
	 \qquad \mbox{ as } \quad r  \to 0^+
			 } 
	holds for all $x\in \Omega$ in the viscosity sense. 	
\end{theorem}}

\section{Gradient dependent operators}
\label{graddepend}
\subsection{The ``nonlocal'' $p$-Laplacian} In this section, we are interested in a nonlocal version of the $p$-Laplace operator, that arises in tug-of war games, and that was introduced in \cite{graddep}. 

This operator is the nonlocal version of the $p$-Laplacian given in a non-divergence form, and deprived of the $|\nabla u|^{p-2}$ factor (namely, the normalized $p$-Laplacian defined in \eqref{renpl}).  
So, for $p\in (1,+\infty)$, the (normalized) $p$-Laplace operator when  $\nabla u\neq 0$ is defined as
  	\[ 
 \Delta^{\mathcal N}_p u:=		\Delta^{\mathcal N}_{p,\pm} u = \Delta u + (p-2) |\nabla u|^{-2}  \langle D^2u \nabla u, \nabla u\rangle 
		.\] 
		By convention, when $\nabla u=0$, as in \cite{graddep}, 
	\[ 
		\Delta^{\mathcal N}_{p,+} u:= \Delta u + (p-2) \sup_{\xi \in \Sf } \langle D^2u \, \xi, \xi\rangle 
		\]
and
\[ 
		\Delta^{\mathcal N}_{p,-} u := \Delta u + (p-2) \inf_{\xi \in \Sf } \langle D^2u\, \xi, \xi\rangle. 
		\]
 
		Let $s\in (1/2,1)$ and $p\in [2,+\infty)$.
		In the nonlocal setting we have the following definition given in \cite[Section 4]{graddep}.
		
		 When $\nabla u(x)= 0$  we define
		 	 \[
		 	( -\Delta)_{p,+}^s u(x) := \frac{1}{\alpha_p}
		 	\sup_{\xi \in \Sf} \int_{\Rn} \frac{2u(x)-u(x+y)-u(x-y)}{|y|^{n+2s}} \chi_{[c_p,1]}\left(\frac{y}{|y|}\cdot \xi\right) dy
		 \]
		 and
 \[
		 	( -\Delta)_{p,-}^s u(x) := \frac{1}{\alpha_p}
		 	\inf_{\xi \in \Sf} \int_{\Rn} \frac{2u(x)-u(x+y)-u(x-y)}{|y|^{n+2s}} \chi_{[c_p,1]}\left(\frac{y}{|y|}\cdot \xi\right) dy.
		 \]
		 When $\nabla u(x) \neq 0$  then
		 \[
		 	( -\Delta)_p^s u(x) :=( -\Delta)_{p,\pm}^s =  \frac{1}{\alpha_p} \int_{\Rn} \frac{2u(x)-u(x+y)-u(x-y)}{|y|^{n+2s}} \chi_{[c_p,1]}\left(\frac{y}{|y|}\cdot z(x) \right) dy
		 \]
		 with 
		 \[
		 	 z(x)= \frac{\nabla u(x)}{|\nabla u(x)|}.
		 \]		
		 Here, $c_p,\alpha_p$ are positive constants.

	We remark that the case $p\in (1,2)$ is defined with the kernel $\chi_{[0,c_p]}\left(\frac{y}{|y|}\cdot z(x)\right)$ for some $c_p>0$, and can be treated in the same way.

		In particular, for $p\in [2,+\infty)$ we consider
		\eqlab{\label{alphp}
			&\alpha_p := \frac12 \int_{\Sf} (\omega \cdot e_2)^2 \chi_{[c_p,1]}(\omega \cdot e_1) \, d \omega,\\
			&\beta_p := \frac12 \int_{\Sf} (\omega \cdot e_1)^2 \chi_{[c_p,1]}(\omega \cdot e_1) \, d \omega - \alpha_p,
			}
			and  
			\eqlab{\label{ccp}
			 c_p \quad \mbox{ such that } \quad \frac{\beta_p}{\alpha_p}=p-2
			.}
			With these constants,  if $u\in C^2(\Rn)\cap L^\infty(\Rn)$, then 
			\[ \lim_{s\to 1^-} (1-s)\Delta_p^s u (x)= \Delta^{\mathcal N}_p u(x),\]
			as proved in \cite[Subsections 4.2.1, 4.2.2]{graddep}.
		
		We define now a $(s,p)$-mean kernel for the nonlocal $p$-Laplacian.
		For any $r>0$ and $u \in L^\infty(\Rn)$, when $\nabla u(x)=0$, we define
		\bgs{
		M_r^{s,p,+} u(x): = \frac{C_{s,p,+} r^{2s}}{2}  \sup_{\xi \in \Sf} \int_{\C B_r} \frac{ u(x+y)+u(x-y)}{|y|^n(|y|^2-r^2)^s } \chi_{[c_p,1]}\left( \frac{y}{|y|}\cdot \xi \right) dy, 
		}
		and
		\bgs{
		M_r^{s,p,-} u(x): = \frac{C_{s,p,-} r^{2s}}{2}  \inf_{\xi \in \Sf} \int_{\C B_r} \frac{ u(x+y)+u(x-y)}{|y|^n(|y|^2-r^2)^s } \chi_{[c_p,1]}\left( \frac{y}{|y|}\cdot \xi \right) dy, 
		}
		with 
		\[ 
			C_{s,p,+} = c_{s} \gamma_{p,+} 
			\quad \mbox{ with }\quad 
			\gamma_{p,+}:=\left(\sup_{\xi\in \Sf} \int_{\Sf} \chi_{[c_p,1]}(\omega \cdot \xi) d\omega\right)^{-1}
		\]
		respectively
			\[ 
			C_{s,p,-} = c_{s} \gamma_{p,-} 
			\quad \mbox{ with }\quad 
			\gamma_{p,-}:=\left(\inf_{\xi\in \Sf} \int_{\Sf} \chi_{[c_p,1]}(\omega \cdot \xi) d\omega\right)^{-1}.
		\]
		When $\nabla u(x) \neq 0$, let
			\bgs{
	M_r^{s,p} u(x): = M_r^{s,p,\pm}u(x)= &\;\frac{C_{s,p} r^{2s}}{2}  \int_{\C B_r} \frac{ u(x+y)+u(x-y)}{|y|^n(|y|^2-r^2)^s } \chi_{[c_p,1]}\left( \frac{y}{|y|}\cdot z(x)\right) dy, 
	\\
	& z(x)= \frac{ \nabla u(x)}{|\nabla u(x)|},
	}
		where{\footnote{It holds that $c(s)= \frac{2\sin \pi s}{\pi}$. }}
		\bgs{
		C_{s,p}= c_{s} \gamma_p , 
				\quad \mbox{ with }\quad
			c_{s}:= \left(\int_1^{\infty} \frac{d \rho}{\rho(\rho^2-1)^s}\right)^{-1}, \quad \gamma_p:=\left(\int_{\Sf} \chi_{[c_p,1]}(\omega \cdot e_1) d\omega\right)^{-1}.
		}

We have the next asymptotic expansion for smooth functions. 
\begin{theorem}\label{mrspw}
Let $\eta>0, x\in\Rn$ and let $u\in C^{2}(B_\eta(x)) \cap L^\infty(\Rn)$. Then 	
\textcolor{black}{	\bgs{ u(x)= M_r^{s,p,\pm} u(x) +
		  c(n,s,p) r^{2s} (-\Delta)_{p,\pm}^s u(x)	  +\mathcal O(r^{2})
	,} }
	as $r  \to 0^+$.	 
\end{theorem}			

\begin{proof}
We prove the result for $\nabla u(x) \neq 0$ (the proof goes the same for $\nabla u(x) =0$).
\\
We fix some $\bar \eps>0$, and there exists $ 0<r= r(\bar \eps)\in(0,\eta/2)$ such that \eqref{bareps} is satisfied. 
Passing to spherical coordinates we have  that
 \bgs{
 		&\; C_{s,p} r^{2s} \int_{\C B_r} \frac{dy}{|y|^n(|y|^2-r^2)^s} \chi_{[c_p,1]}\left( \frac{y}{|y|} \cdot z(x) \right)=
 		C_{s,p}\int_1^{\infty} \frac{ d\rho}{\rho(\rho^2-1)^s} \int_{\Sf} \chi_{[c_p,1]}\left( \omega \cdot z(x) \right)d \omega 		\\
 		=&\; C_{s,p} \int_1^{\infty} \frac{ dy}{\rho(\rho^2-1)^s} \int_{\Sf} \chi_{[c_p,1]}\left( \omega \cdot e_1 \right) d \omega=
 		1,
 	}
 	where the last line follows after a rotation (one takes $U\in \M^{n\times n}(\R)$ an orthogonal matrix such that $U^{-1}(x) z(x)= e_1$ and changes variables).

 	It follows that for any $r>0$,
 	\[ u(x) - M_r^{s,p} u(x) = \frac{C_{s,p} r^{2s}}2 \int_{\C B_r} \frac{ 2u(x)-u(x+y)-u(x-y) }{|y|^n(|y|^2-r^2)^s } \chi_{[c_p,1]}\left( \frac{y}{|y|}\cdot z(x)\right)dy. \] 
 	Therefore we have that
 	\bgs{
 	  & u(x) - M_r^{s,p} u(x) =\;  \frac{C_{s,p} \alpha_p r^{2s}}2 (-\Delta)_p^s u(x) 
 	 \\
 	 &\; - \frac{C_{s,p} r^{2s}}2 \int_{B_r}  \frac{ 2u(x)-u(x+y)-u(x-y) }{|y|^{n+2s} } \chi_{[c_p,1]}\left( \frac{y}{|y|}\cdot z(x)\right)dy
 	 \\
 	 &\; + \frac{C_{s,p} r^{2s}}2 \int_{\C B_r} \frac{ 2u(x)-u(x+y)-u(x-y) }{|y|^{n+2s} } \left( \frac{|y|^{2s}}{(|y|^2-r^2)^s}-1 \right)  \chi_{[c_p,1]}\left( \frac{y}{|y|}\cdot z(x)\right)dy
 	 \\
 	 = :&\; 	\frac{C_{s,p} \alpha_p r^{2s}}2 (-\Delta)_p^s u(x) 
 	 -I_r+J_r
 	 }
 	and
	\bgs{
	J_r= &\; \frac{C_{s,p} }2 \int_{\C B_1}\frac{ 2u(x)-u(x+ry)-u(x-ry) }{|y|^{n+2s} } \left( \frac{|y|^{2s}}{(|y|^2-1)^s}-1 \right)  \chi_{[c_p,1]}\left( \frac{y}{|y|}\cdot z(x)\right)dy
	\\
	= &\; \frac{C_{s,p} }2 \int_{B_\frac{\eta}r\setminus B_1} \frac{ 2u(x)-u(x+ry)-u(x-ry) }{|y|^{n+2s} } \left( \frac{|y|^{2s}}{(|y|^2-1)^s}-1 \right)  \chi_{[c_p,1]}\left( \frac{y}{|y|}\cdot z(x)\right)dy
	\\
	&\; + \frac{C_{s,p} }2 \int_{\C B_\frac{\eta}r} \frac{ 2u(x)-u(x+ry)-u(x-ry) }{|y|^{n+2s} } \left( \frac{|y|^{2s}}{(|y|^2-1)^s}-1 \right)  \chi_{[c_p,1]}\left( \frac{y}{|y|}\cdot z(x)\right)dy
	\\
	= &\; J^1_r+J^2_r.
	}
	We obtain that
	\bgs{
		|J^2_r| \leq &\; 4\|u\|_{L^\infty(\Rn)}  \frac{C_{n,s,p} }2 \int_{\frac{\eta}r}^\infty \frac{d\rho }{\rho^{1+2s} } \left( \frac{\rho^{2s}}{(\rho^2-1)^s}-1 \right) \int_{\Sf} \chi_{[c_p,1]}\left( \omega\cdot z(x)\right)d\omega
		\\
		\leq &\; C_{s,p}  \int_{\frac{\eta}r}^\infty \frac{d\rho }{\rho^{1+2s} } \left( \frac{\rho^{2s}}{(\rho^2-1)^s}-1 \right) ,
		}
		and using \eqref{trois}, that 
		\[ J^2_r= \mathcal O (r^{2+2s}).\]

 	We have that
 	 \bgs{
 	 	J_r^1-I_r = &\; \frac{C_{s,p}}2  
 	 	\bigg[ \int_{B_{\frac{\eta}r} \setminus B_1 }  \frac{ 2u(x)-u(x+ry)-u(x-ry) }{|y|^n (|y|^2-1)^s} \chi_{[c_p,1]}\left( \frac{y}{|y|}\cdot z(x)\right) dy
 	 	\\
 	 	&\; - \int_{B_{\frac{\eta}r} }  \frac{ 2u(x)-u(x+ry)-u(x-ry) }{|y|^{n+2s} } \chi_{[c_p,1]}\left( \frac{y}{|y|}\cdot z(x)\right) dy \bigg]		 	
 	 	}
 which, by \eqref{11} and  \eqref{eun}, gives 
		\[ J_r^1 -I_r=\mathcal O(r^{2}).\]
		It follows that 
		\[u(x) - M_r^{s,p} u(x) =\frac{C_{s,p} \alpha_p }2\: r^{2s} (-\Delta)_p^s u(x) 
		+\mathcal O(r^{2})
		\]
		for $r  \to 0^+$, hence the conclusion.
		\end{proof}	
				
We recall  the viscosity setting introduced in \cite{graddep}.
\begin{definition} A function $u\in L^\infty(\Rn)$, upper (lower) semi-continuous in $\overline \Omega$ is a viscosity subsolution  (supersolution) in $\Omega$ of 
	\[
		(-\Delta)_{p,\pm}^s u=0, \qquad \mbox{ and we write } \quad (-\Delta)_{p,\pm}^s u \leq \,(\geq) \,  0
	\]
	if for every $x\in\Omega$, any neighborhood $U=U(x)\subset \Omega$ and any $\varphi \in C^2(\overline U)$ such that 
	\eqref{visc23}  holds
	if we let $v$ as in \eqref{vvv234}
				 \bgs{
				 	(-\Delta)_{p,\pm}^s v (x) \leq \,(\geq)\,0.
				 }			
	A viscosity solution of $(-\Delta)_{p,\pm}^s u =0$ is a (continuous) function that is both a subsolution and a supersolution. 
	\end{definition}	
	Furthermore, we define an asymptotic expansion in the viscosity sense. 
	\begin{definition} Let $u \in L^\infty(\Rn)$ upper (lower) semi-continuous in $\Omega$. We say that 
\bgs{
\label{fggdw}
	\lim_{r \to 0^+}  \left(u(x)- M^{s,p}_r u(x)\right)  = o(r^{2s})  
	}
	holds in the viscosity sense
if for any 
	neighborhood $U=U(x)\subset \Omega$ and any $\varphi \in C^2(\overline U)$ such that \eqref{visc23} holds, and
if we let $v$ be defined as in \eqref{vvv234}, then both
\bgs{
	\liminf_{r \to 0^+}  \frac{ u(x)- M^{s,p}_r u(x)}{r^{2s}}   \geq 0 
	}
	and
	\bgs{
	\limsup_{r \to 0^+}   \frac{ u(x)- M^{s,p}_r u(x)}{r^{2s}}    \leq 0 
	}
	hold point wisely.
\end{definition}	

The result for viscosity solutions, which is a direct  consequence of Theorem \ref{mrspw} applied to the test function $v$, goes as follows.
\begin{theorem}\label{theorem2} Let $\Omega\subset \Rn$ be an open set and let $u\in C(\Omega)\cap L^{\infty}(\Rn)$. Then 
\[ 
	(-\Delta)_{p,\pm}^s u(x) =0 
	\]
in the viscosity sense if and only if 
	\bgs{
	\lim_{r \to 0^+}  \left(u(x)- M^{s,p,\pm}_r u(x)\right)  = o(r^{2s})  
	}
	holds for all $x\in \Omega$ in the viscosity sense. 	
\end{theorem}

We study also the limit case as $s\to 1^-$ of this version of the $(s,p)$-mean kernel. We state the result only in the case $\nabla u(x)\neq 0$, remarking that an analogue results holds for $M_r^{s,p,\pm}$ with the suitable $M_r^{p,\pm}$. 
\begin{proposition} Let $\Omega\subset \Rn$ be an open set and  $u\in C^{1}(\Omega) \cap L^\infty(\Rn)$. For any $r>0$ small 
denoting
\[ M_r^p u(x):= \frac{\textcolor{black}{\gamma_p}}2 \int_{\partial B_r} \left(u(x+y)-u(x-y)\right) 
\chi_{[c_p,1]} \left( \frac{y}{|y|} \cdot z(x) \right) \, dy \]
it holds that
\eqlab{\label{2w} \lim_{s\to 1^-} M_r^{s,p} u(x) = M_r^p u(x),}
	for every $x\in \Omega, r>0$ such that $B_{2r}(x)\subset \Omega$.
\end{proposition} 
\begin{proof}
We have that
	\bgs{
		M_r^{s,p} u(x) = \frac{C_{s,p}}2 \int_{\C B_1} \frac{ u(x+ry) + u(x-ry)}{|y|^n (|y|^2-1)^s} \chi_{[c_p,1]}\left( \frac{y}{|y|}\cdot z(x)\right) dy .
	}
	Let $\eps>0 $ be fixed (to be taken arbitrarily small). Then
	\[
		|J_\eps(x)|:= \left| \int_{B_{1+\eps}}\frac{ u(x+ry) + u(x-ry)}{|y|^n (|y|^2-1)^s} \chi_{[c_p,1]}\left( \frac{y}{|y|}\cdot z(x)\right) dy \right| \leq \frac{ 2\|u\|_{L^\infty(\Rn)}}{\gamma_p} \int_{1+\eps}^\infty \frac{dt}{t(t^2-1)^s},    
	\] 
	which from Proposition \ref{useful} gives that
	\[ \lim_{s\to 1^-}C_{s,p} J_\eps(x) =0.\] 
	On the other hand, we have that
	\bgs{
	I_\eps(x) =&\; \int_{B_{1+\eps}\setminus B_1} \frac{ u(x+ry) + u(x-ry)}{|y|^n (|y|^2-1)^s} \chi_{[c_p,1]}\left( \frac{y}{|y|}\cdot z(x)\right) dy 
	\\
	=&\; \int_{\Sf} \left( \int_1^{1+\eps} \frac{ u(x+r\rho \omega ) + u(x-r\rho \omega)}{\rho (\rho^2-1)^s}\, d\rho \right) \chi_{[c_p,1]}\left( \omega\cdot z(x)\right) d\omega 
	}
	and integrating by parts, that
	\bgs{
		\int_1^{1+\eps} \frac{ u(x+r\rho \omega ) + u(x-r\rho \omega)}{\rho (\rho^2-1)^s}\, d\rho 
		= &\;  \frac{\eps^{1-s}}{1-s} \frac{ u(x+r(1+\eps) \omega ) + u(x-r(1+\eps) \omega)}{(1+\eps) (2+\eps)^s} 
		 - I^o_\eps(x)
		}
		with
		\bgs{ 
		I^o_\eps(x):= \int_1^{1+\eps} \frac{(\rho-1)^{1-s}}{1-s}  \frac{d}{d\rho} \left(\frac{ u(x+r\rho \omega ) + u(x-r\rho \omega)}{\rho (\rho+1)^s}\right) d\rho.		
	}
	We notice that
	\[
		 \left| I^o_\eps(x)\right| \leq C \frac{\eps^{2-s}}{1-s} ,
	\]
	hence we get 
	\[ \lim_{s\to 1^-} C_{s,p} I_\eps^o(x) = \mathcal O(\eps).\] 
	Therefore we obtain
	\bgs{
			 \lim_{s \to 1^-} M_r^{s,p} u(x) = \frac{\textcolor{black}{\gamma_p}}{(1+\eps)(2+\eps)}  \int_{\Sf}  \left( u(x+r\omega) +u(x-r\omega)\right)\chi_{[c_p,1]}\left( \omega \cdot z(x)\right) dy + \mathcal O(\eps),
		}
		and \eqref{2w} follows by sending $\eps \to 0$.
\end{proof}
\textcolor{black}{We obtain furthermore an expansion for the normalized $p$-Laplacian, as follows.
\begin{proposition}
 \label{mvploc122}
If $u\in C^2(\Omega)$, then
	\[
		u(x) -M_r^pu(x)=- C_p r^2 \Delta_p^{\mathcal N} u(x) + o(r^2).
	\]
\end{proposition}	
	\begin{proof}
	Using the Taylor expansion in \eqref{11} with a Peano remainder, we have that
	\bgs{
		u(x)- M_r^p u(x)= &\; \frac{\gamma_p}2 \int_{\Sf} \left(2u(x)-u(x-r\omega)-u(x+r\omega)\right)\chi_{[c_p,1]}\left(\omega\cdot z(x) \right)d\omega
		\\
		=&\; - \frac{\gamma_pr^2}2 \int_{\Sf} \left(\langle D^2u(x) \omega, \omega\rangle\right) \chi_{[c_p,1]}\left(\omega\cdot z(x) \right)d\omega + o(r^2).
		} 
		As in \cite[Subsections 4.2.1, 4.2.2]{graddep}, it holds that
		\[
		\int_\Sf \langle D^2u(x) \omega, \omega\rangle \chi_{[c_p,1]}\left(\omega\cdot z(x) \right)d\omega = 2\alpha_p  \Delta_p^{\mathcal N} u(x),
		\]
		and the conclusion immediately follows.
	\end{proof}
	An analogue result holds for the suitable $M_r^{p,\pm}$, and the same we can say about the 
 the next theorem in the viscosity setting (which follows from the asymptotic expansion for smooth functions). 
\begin{theorem}\label{theorem234567} Let $\Omega\subset \Rn$ be an open set and let $u\in C(\Omega)\cap L^{\infty}(\Rn)$. Then 
\[ 
	(-\Delta)_p^{\mathcal N} u(x) =0 
	\]
in the viscosity sense if and only if 
	\bgs{
	 			 		\lim_{r \to 0^+} \big(u(x)- M^{p}_r u(x)\big)=
	 			 		o(r^2)
	 \qquad \mbox{ as } \quad r  \to 0^+
			 } 
	holds for all $x\in \Omega$ in the viscosity sense. 	
\end{theorem}}

\subsection{The infinity fractional Laplacian}\label{9876}
In this section, we deal with the infinity fractional Laplacian, arising in a nonlocal tug-of-war game, as introduced in \cite{bjor}. Therein, the authors deal with viscosity solutions of a Dirichlet monotone problem and a monotone double obstacle problem, providing a comparison principle on compact sets and H\"{o}lder regularity of solutions.

The infinity Laplacian in the non-divergence form is defined by omitting the term $|\nabla u|^2$, precisely when $\nabla u(x)=0$,
	\bgs{
		&\Delta_{\infty,+} u(x):= \sup_{\xi \in \Sf } \langle D^2u(x) \, \xi, \xi\rangle 
		\quad 
		\Delta_{\infty,-}u(x) := \inf_{\xi \in \Sf } \langle D^2u(x)\, \xi, \xi\rangle.}
		\textcolor{black}{and formally
		\bgs{\Delta_\infty u(x) :=\frac{ \Delta_{\infty,+} u(x)+ \Delta_{\infty,-} u(x)}2,
		}}
		whereas when  $\nabla u(x)\neq 0$,
	\[ 
		\Delta_{\infty} u(x):=\Delta_{\infty,\pm} u(x)=  \langle D^2u(x) \, z(x), z(x)\rangle, \quad \mbox{ where } \quad z(x) = \frac{\nabla u(x)}{|\nabla u(x)|}. 
		\]

The definition in the fractional case is well posed for $s\in (1/2,1)$, given in \cite[Definition 1.1]{bjor}. Let $s\in (\frac 12, 1)$. The infinity fractional Laplacian 
is defined in the following way:
\begin{itemize}
\item If $\nabla u(x)\neq 0$ then
	\eqlab{ \label{opneq}
		 (-\Delta)^s_\infty u(x):=\int_0^\infty \frac{ 2u(x) -u(x+\rho z(x)) -u(x-\rho z(x)) }{\rho^{1+2s}} d\rho ,
		}  
		where $ z(x)=\frac{\nabla u(x)}{|\nabla u(x)|}\in \Sf$.
\item If $\nabla u(x)=0$ then
\eqlab{ \label{opeq}
		 (-\Delta)^s_\infty u(x):=\sup_{\omega \in \Sf} \inf_{\zeta \in \Sf} \int_0^\infty \frac{ 2u(x) -u(x+\rho \omega)  -u(x-\rho \zeta) }{\rho^{1+2s}} d\rho .
		}  

	\end{itemize}
	
		There exist ``infinity harmonic functions'':  it is proved in  \cite{bjor} that the function 
		\[ C(x)=A|x-x_0|^{2s-1} +B\]
		satisfies 
		\[ (-\Delta)^s_\infty u(x)= 0\quad  \mbox{ for any } x\neq x_0.\]    
		
	We denote
		\[\L u(x,\omega,\zeta) :=\int_0^\infty \frac{2u(x) -u(x+\rho \omega) -u(x-\rho \zeta)}{\rho^{1+2s}}\, d\rho \]
and for $r>0$
	\bgs{
		M^s_r u(x,\omega,\zeta):=
		 c_s r^{2s} \int_r^\infty \frac{u(x+\rho \omega) + u(x-\rho \zeta)}{(\rho^2-r^2)^s\rho} d\rho ,		
	}
	with
	\[ c_{s}:= \textcolor{black}{\frac12}\left(\int_1^{\infty} \frac{d \rho}{\rho(\rho^2-1)^s}\right)^{-1} = \frac{\sin \pi s}{\pi}.
	\]
	We define the operators

	\begin{itemize}
	\item if $\nabla u(x)\neq 0$
	\[ \M_{r}^{s,\infty} u(x): = M_r^s u(x,z(x),z(x)), \qquad \mbox{ with } z(x)=\frac{\nabla u(x)}{|\nabla u(x)|},\]
	\item  if $\nabla u(x)= 0$
	\[ \M_r^{s,\infty} u(x):= \sup_{\omega\in \Sf} \;  \inf_{  \zeta \in \Sf}  M_r^s u(x,\omega,\zeta).
	 \]    
	\end{itemize}


We obtain the asymptotic mean value property for smooth functions, as follows.
\begin{theorem}\label{asympw}
Let $\eta>0, x\in\Rn$ and let $u\in C^{2}(B_\eta(x)) \cap L^\infty(\Rn)$. Then 	
	\bgs{ u(x)= \M_r^{s,\infty} u(x) +
		  c(s) r^{2s} \frin u(x)	  +\textcolor{black}{\mathcal O(r^{2})} 
	}
	as $r\to 0^+$.
\end{theorem}
	
\begin{proof}
We have that
	\bgs{\label{first} 
	u(x)- M_r^s u(x,\omega,\zeta) 
		  =&\, c_s r^{2s} \int_r^\infty   \frac{2u(x) -u(x+\rho \omega)-u(x-\rho \zeta)}{\rho(\rho^2 -r^2)^s } d\rho,
	  }
	  hence
	  \bgs{
	  	u(x)- M_r^s u(x,\omega,\zeta)  	=&\;  c_s \bigg[ r^{2s} \L u(x,\omega,\zeta) 
	  	- \int_{B_1}  \frac{2u(x) -u(x+r \rho \omega)-u(x-r \rho \zeta)}{\rho^{1+2s} } d\rho 
	  	\\
	  	&\;+ \int_{\C B_1}  \frac{2u(x) -u(x+r \rho \omega)-u(x-r \rho \zeta)}{\rho^{1+2s} } \left(\frac{\rho^{2s}}{(\rho^2-1)^{2s} }-1 \right) d\rho \bigg]
	  	\\
	  	=: &\;   c_s \left( r^{2s} \L u(x,\omega,\zeta)  - I_r +J_r\right).
	  	 }
Then 
	\bgs{
		J_r=&\;  \int_{B_{\frac{\eta}r\setminus  B_1}}  \frac{2u(x) -u(x+r \rho \omega)-u(x-r \rho \zeta)}{\rho^{1+2s} } \left(\frac{\rho^{2s}}{(\rho^2-1)^{2s}  }-1 \right) d\rho  
		\\
		&\;+\int_{\C B_{\frac{\eta}r} } \frac{2u(x) -u(x+r \rho \omega)-u(x-r \rho \zeta)}{\rho^{1+2s} } \left(\frac{\rho^{2s}}{(\rho^2-1)^{2s} }-1 \right) d\rho 
		\\
		=:&\; J_r^1+J_2^r.
	}
We proceed as in the proof of Theorem \ref{mrspw}, using also \eqref{11} and Proposition \ref{useful}, and obtain that
	\[
	J^2_r= \mathcal O(r^{2s+2s}) \qquad \mbox{and} \qquad J^1_r-I_r= \textcolor{black}{\mathcal O(r^{2})} .
	\]
This concludes the proof of the Theorem.
\end{proof}

		The main result of this section, which follows from Theorem \ref{asympw}, is stated next.
		
\begin{theorem}\label{theorem1} Let $\Omega\subset \Rn$ be an open set and let $u\in C(\Omega)\cap L^{\infty}(\Rn)$. The asymptotic expansion
	\eqlab{ \label{eq1}
	 u(x)=  \M^{s,\infty}_r u(x) 
	 +o(r^{2s}), \qquad \mbox{ as } r\to 0
	 } 
	holds for all $x\in \Omega$ in the viscosity sense if and only if
	\[ 
	\frin u(x) =0 
	\]
in the viscosity sense.
\end{theorem}

We investigate also the limit case $s\to 1^-$. 
 	
\begin{proposition}
 Let $\Omega\subset \Rn$ be an open set and  $u\in C^{1}(\Omega) \cap L^\infty(\Rn)$. Then 
	\sys[ 		\lim_{s\to 1^-}  \M^s_r u(x) = \M_r^\infty u(x) :=] 
	{&\frac12\left(u\big(x+rz(x)\big)+u\big(x-rz(x)\big)\right) && \mbox{ when } \nabla u(x)\neq 0,\\
&\frac12\left(  \sup_{\omega \in \Sf} u(x+r\omega) + \inf_{\zeta\in\Sf}u(x-r\zeta)\right) && \mbox{ when } \nabla u(x)= 0,}
	for every $x\in \Omega, r>0$ such that $B_{2r}(x)\subset \Omega$.
\end{proposition}

\begin{proof}
For some $\eps>0$ small enough, we have that
 	\eqlab{ \label{opium}
 		M_r^s u(x,\omega ,\zeta)  =&\,	c_s\left( \int_{1+\eps}^\infty \frac{u(x+r\rho \omega)+ u(x-r\rho \zeta)}{(\rho^2-1)^s\eta}\, d\rho  + \int_1^{1+\eps} \frac{u(x+r\rho \omega)+ u(x-r\rho \zeta)}{(\eta^2-1)^s\rho}\, d\rho \right) 			 \\
			 =:&\,  I_s^1 +  I_s^2.
  	}
Using Proposition \ref{useful}, we get that
	\[
	\lim_{s\to 1}c_s  I_s^1=0 .
	\]
		Integrating by parts in $I_s^2$, we have
		\bgs{ 
		\left|\int_1^{1+\eps}\frac{u(x + r\rho \omega )}{(\eta^2-1)^s\rho}   d\rho  
				- \frac{\eps^{1-s} u\left(x + r(1+\eps)\omega\right)}{(1-s)(\eps+2)^s (1+\eps)}\right|
				\leq 				C\frac{\eps^{2-s} }{1-s} ,
	}
thus
		\bgs{
		 \bigg| \mathcal I_s^2 -&\,   
		 \frac{\eps^{1-s}}{(1-s)(\eps+2)^s (1+\eps)}
	  \big( u\left(x +r(1+\eps)\omega\right) +u\left(x-r(1+\eps)\zeta \right)\big)  \bigg|
		\leq  C \frac{\eps^{2-s}}{1-s}  .
	}
	We get that
	\bgs{
		\lim_{s\to 1^-} c_s \mathcal I_s^2 = \frac{1}{ (\eps +2)(\eps+1) } \big( u(x+r(1+\eps)y)+u(x-r(1+\eps)z) \big) \, + C \eps.
	}
	Sending $\eps \to 0$ we get the conclusion.
 \end{proof}

For completeness, we show the following, already known, result.

\begin{proposition}
Let $u\in C^{2}
(\Omega) \cap L^\infty(\Rn)$. For all $x\in \Omega$ for which $|\nabla u (x)|\neq 0$ it holds that 
\[ 		\lim_{s\to 1^-}  (1-s) (-\Delta)^s_\infty u(x) = -\Delta_\infty u(x).\]
\end{proposition}

\begin{proof}
		Since $u\in C^2(\Omega)$ we have that for any $\bar \eps>0$ there exists $ r= r(\bar \eps)>0$ such that 
	\eqref{bareps} holds. 
	We prove the result for $\nabla u(x) \neq 0$ (the other case can be proved in the same way).  We have that
	\bgs{
	(-\Delta)^s_\infty=&\; \int_0^r \frac{2u(x) -u\left(x+\rho z(x)\right) -u\left(x-\rho z(x)\right)}{\rho^{1+2s}}\, d\rho 
	\\
	 &\;+ \int_r^\infty \frac{2u(x) -u\left(x+\rho z(x)\right) -u\left(x-\rho z(x)\right)}{\rho^{1+2s}}\, d\rho 
	= I_r +J_r. 
	}
	We have that
	\[
	 |J_r| \leq C\|u\|_{L^\infty(\Rn)} \frac{r^{-2s}}{2s},
	 \qquad \mbox{ and } \qquad 
		\lim_{s \to 1^-} (1-s) J_r =0.
		\] 	
		On the other hand, using \eqref{11} we have that
		\bgs{
		I_r =&\;   -\int_0^r \frac{ \langle D^2u(x) z(x), z(x) \rangle } \rho^{1-2s} \, d\rho + I^o_r
		=  -\langle D^2u(x) z(x), z(x) \rangle \frac{r^{2-2s}}{2(1-s)} + I^o_r ,
		}
		with
		\[ 
		\lim_{s\to 1^-} (1-s) I^o_r = \mathcal O(\bar \eps).
		\]
		The conclusion follows by sending $\bar \eps \to 0$.
	\end{proof}
	\textcolor{black}{We mention that the mean value property for the infinity Laplacian is settled in \cite{MPR}. For the sake of completeness, we however write the very simple expansion for the infinity Laplacian.
	\begin{proposition}
 \label{mvploc189}
If $u\in C^2(\Omega)$, then
	\[
		u(x) -M_r^\infty u(x)=- c r^2 \Delta_\infty u(x) + o(r^2).
	\]
\end{proposition}
An immediate consequence is the following theorem in the viscosity setting.
\begin{theorem}\label{theorem234589} Let $\Omega\subset \Rn$ be an open set and let $u\in C(\Omega)\cap L^{\infty}(\Rn)$. Then 
\[ 
	(-\Delta)_\infty u(x)=0
	\]
in the viscosity sense if and only if 
	\bgs{
	 			 		\lim_{r \to 0^+} \big(u(x)- M^{\infty}_r u(x)\big)=
	 			 		o(r^2)
	 \qquad \mbox{ as } \quad r  \to 0^+
			 } 
	holds for all $x\in \Omega$ in the viscosity sense. 	
\end{theorem}}
 \appendix
 \section{Useful asymptotics}\label{appendicite}

We insert in this appendix some asymptotic results, that we use along the paper. 
\begin{proposition}\label{useful}
Let $s\in (0,1)$. For $r$ small enough the following hold:
	\begin{subequations}
	\begin{align}
		& \int_1^{\frac{1}{r}} t \left(\frac{1}{(t^2-1)^s}-\frac{1}{t^{2s}} \right) \, dt = \mathcal O(1), \label{eun}
	\\
	& \int_{\frac{1}{r}}^\infty \frac{1}{t} \left( \frac{t^{2s}}{(t^2-1)^s} - 1\right) \, dt = \mathcal O(r^{2}). \label{trois}
	\end{align}
	\end{subequations}
	Furthermore,
	\eqlab{\label{qutr} \lim_{s\to 1^-} (1-s)\int_{1+r}^\infty \frac{dt}{t(t^2-1)^s} \, dt =0.
	}
\end{proposition}
\begin{proof}
To prove \eqref{eun}, integrating, we have that 
\bgs{
		 \int_1^{\frac{1}{r}} t \left(\frac{1}{(t^2-1)^s}-\frac{1}{t^{2s}} \right) \, dt  = \frac{1}{2(1-s)} \left( \frac{(1-r^2)^{1-s} -1}{r^{2(1-s)} } +1 \right) = \frac{1}{2(1-s)} \left( \mathcal O(r^{2s})+1 \right).
		 }
		 In a similar way, we get \eqref{eun}. To obtain \eqref{trois}, we notice that since $\frac{1}{t}<r<1$, with a Taylor expansion we have
	\[ 
		 \frac{1}{(1-\frac{1}{t^2})^s}-1
		= s\frac{1}{t^2} +  o \left(\frac{1}{t^2} \right),
	\]
	and the conclusion is reached by integrating. 
	Furthermore
	\bgs{
	\int_{1+r}^\infty \frac{dt}{t(t^2-1)^s} \, dt
	= \int_{1+r}^2 \frac{dt}{t(t^2-1)^s} \, dt +\int_{2}^\infty \frac{dt}{t(t^2-1)^s} \, dt \leq \frac{ c(1-r^{1-s})}{1-s} + \frac{c}{s}.
	}
	Multiplying by $(1-s)$ and taking the limit, we get \eqref{qutr}.	
\end{proof}

\bigskip
\bibliography{biblio}
\bibliographystyle{plain}

\end{document}